\documentclass[reqno,10pt,centertags, draft]{amsart}
\usepackage{amsmath,amsthm,amscd,amssymb,latexsym,upref,mathrsfs}

\newcommand{\bbN}{{\mathbb{N}}}
\newcommand{\bbR}{{\mathbb{R}}}

\newcommand{\bbC}{{\mathbb{C}}}

\newcommand{\dC}{{\mathbb{C}}}
\newcommand{\dR}{{\mathbb{R}}}

\newcommand{\cA}{{\mathcal A}}
\newcommand{\cB}{{\mathcal B}}
\newcommand{\cC}{{\mathcal C}}
\newcommand{\cD}{{\mathcal D}}

\newcommand{\cG}{{\mathcal G}}
\newcommand{\cH}{{\mathcal H}}

\newcommand{\cL}{{\mathcal L}}

\newcommand{\cS}{{\mathcal S}}

\newcommand{\cX}{{\mathcal X}}

\newcommand{\sH}{{\mathfrak H}}

\newcommand{\no}{\notag}
\newcommand{\lb}{\label}

\newcommand{\ol}{\overline}

\newcommand{\wti}{\widetilde}

\newcommand{\f}{\frac}
\newcommand{\bi}{\bibitem}

\def\senki{{\lbrack\negthinspace [\bot ]\negthinspace\rbrack}}
\def\senki+{{\lbrack\negthinspace [+] \negthinspace\rbrack}}

\renewcommand{\Im}{\mathop\mathrm{Im}}
\renewcommand{\ge}{\geqslant}

\DeclareMathOperator{\dom}{dom}
\DeclareMathOperator{\ran}{ran}

\allowdisplaybreaks \numberwithin{equation}{section}

\newtheorem{theorem}{Theorem}[section]

\newtheorem{lemma}[theorem]{Lemma}
\newtheorem{corollary}[theorem]{Corollary}
\newtheorem{definition}[theorem]{Definition}
\newtheorem{hypothesis}[theorem]{Hypothesis}

\theoremstyle{remark}
\newtheorem{remark}[theorem]{Remark}

\begin{document}

\title[The Third Green Identity]{Coupling of symmetric operators
and the third Green Identity}

\author[J.\ Behrndt]{Jussi Behrndt}
\address{Institut f\"ur Numerische Mathematik, Technische Universit\"at
Graz, Steyrergasse 30, 8010 Graz, Austria}
\email{behrndt@tugraz.at}
\urladdr{www.math.tugraz.at/~behrndt/}

\author[V.\ Derkach]{Vladimir Derkach}
\address{Department of Mathematics, Dragomanov National Pedagogical University,
Kiev, Pirogova 9, 01601, Ukraine}
\email{derkach.v@gmail.com}

\author[F.\ Gesztesy]{Fritz Gesztesy}
\address{Department of Mathematics,
University of Missouri, Columbia, MO 65211, USA}
\email{gesztesyf@missouri.edu}
\urladdr{http://www.math.missouri.edu/personnel/faculty/gesztesyf.html}
\address{Address after August 1, 2016: Department of Mathematics
Baylor University, One Bear Place \#97328,
Waco, TX 76798-7328, USA}
\email{Fritz$\_$Gesztesy@baylor.edu}


\author[M.\ Mitrea]{Marius Mitrea}
\address{Department of Mathematics,
University of Missouri, Columbia, MO 65211, USA}
\email{mitream@missouri.edu}
\urladdr{http://www.math.missouri.edu/personnel/faculty/mitream.html}

\thanks{J.\,B.\ gratefully acknowledges financial support by the Austrian
Science Fund (FWF), project P 25162-N26; M.\,M.\ is indebted to the Simons Foundation for support under Grant $\#$\,281566; \,
V.D. is indebted to the Fulbright Fund for support under Grant $\#$\, 68130028}


\date{\today}


\begin{abstract}
The principal aim of this paper is to derive an abstract form of the third Green identity
associated with a proper extension $T$ of a symmetric operator $S$ in a Hilbert space $\mathfrak H$, employing the technique of quasi boundary triples for $T$.
The general results are illustrated with couplings of Schr\"{o}dinger operators on Lipschitz domains on smooth, boundaryless Riemannian manifolds.
\end{abstract}

\maketitle

{\scriptsize \tableofcontents}

\section{Introduction}  \lb{s1}

The origin of this paper can be traced back to the following innocent question: {\it How to rule
out that a Dirichlet eigenfunction of the Laplacian, or, more generally, a uniformly elliptic
second order partial differential operator on a nonempty, open, bounded domain
$\Omega \subset \bbR^n$, $n \in \bbN$, with sufficiently regular boundary, is also simultaneously a Neumann eigenfunction?}

There are, of course, several immediate answers. For instance, in the case of the Dirichlet Laplacian on a sufficiently regular, open, bounded, domain $\Omega \subset \bbR^n$,
\begin{equation}
- \Delta u = \lambda u, \quad u \upharpoonright_{\partial\Omega} = 0, \; u \in H^2(\Omega),
\end{equation}
Rellich's identity from \cite{Re40} for (necessarily real) Dirichlet eigenvalues $\lambda$ reads,
\begin{equation}
\lambda = \f{1}{4 \|u\|^2_{L^2(\Omega)}} \int_{\partial\Omega}
\bigg(\f{\partial u}{\partial \nu}(\xi)\bigg)^2 \bigg(\f{\partial |x|^2}{\partial \nu}(\xi)\bigg)
d^{n-1}\omega (\xi).
\end{equation}
Here $x = (x_1,\dots,x_n) \in \bbR^n$ lies in a neighborhood of $\partial \Omega$,
$d^{n-1} \omega$ denotes the surface measure on $\partial \Omega$, $\nu$ is the outward pointing
unit normal vector at points of $\partial \Omega$, and
$\partial / \partial \nu$ represents the normal derivative,
\begin{equation}
 \frac{\partial}{ \partial \nu} := \nu(\xi) \cdot \nabla_{\xi},
\quad \xi \in \partial \Omega.
\end{equation}
Thus, vanishing of the normal derivative (i.e., the Neumann boundary condition)
$\partial u/\partial \nu|_{\partial\Omega} = 0$ yields $\lambda = 0$ which contradicts the well-known
fact that the Dirichlet Laplacian is strictly positive on bounded (in fact, finite Euclidean volume) domains $\Omega$. (More general domains 
such as Lipschitz could be discussed in the context of the examples mentioned in this introduction, but for brevity we stick to 
sufficiently regular, say, $C^2$-domains,  throughout. We will, however, consider Lipschitz domains on a smooth, boundaryless manifold in Section \ref{s5}.)

A second approach, based on
\begin{equation}
\mathring H^2(\Omega) = \big\{v \in H^2(\Omega) \, \big| \, u \upharpoonright_{\partial\Omega} = (\partial_{\nu} u) \upharpoonright_{\partial\Omega} = 0\big\},
\end{equation}
again for $\Omega$ sufficiently regular, shows that
\begin{equation}
- \Delta u = \lambda u, \quad u \upharpoonright_{\partial\Omega}
= (\partial_{\nu} u) \upharpoonright_{\partial\Omega} = 0, \; u \in H^2(\Omega),    \lb{1.5}
\end{equation}
cannot have any nonzero solution $u$ as the zero extension $\wti u$ of $u$ outside
$\Omega$ lies in $H^2(\bbR^n)$ and hence $- \Delta$ on $\bbR^n$ would have a compactly supported eigenfunction, clearly a contradiction.
This extends to more general uniformly elliptic second order partial differential operators via unique continuation principles, see, for instance, \cite{Ar57}, \cite{AKS62}, \cite{F01}, \cite{JK85}, \cite{KT01}, \cite{KT02}, and \cite{Wo93}.

An approach, intimately related to the second approach, adding a functional analytic flavor, would employ the fact that the Dirichlet and Neumann Laplacians in
$L^2(\Omega)$, denoted by $-\Delta_{D,\Omega}$ and $-\Delta_{N,\Omega}$, respectively, are relatively prime and hence satisfy
\begin{equation}
\dom(-\Delta_{D,\Omega}) \cap \dom(-\Delta_{N,\Omega}) = \mathring{H}^2(\Omega)
= \dom(-\Delta_{min,\Omega}).
\end{equation}
Here the associated minimal and maximal operators in $L^2(\Omega)$ are of the form
\begin{align}
& -\Delta_{min,\Omega} = -\Delta, \quad \dom(-\Delta_{min,\Omega}) = \mathring{H}^2(\Omega), \\
& -\Delta_{max,\Omega}:=-\Delta,\quad
\dom(-\Delta_{max,\Omega}):=\big\{f\in L^2(\Omega)\,\big|\,\Delta f\in L^2(\Omega)\big\},
\end{align}
where the expression $\Delta f$, $f\in L^2(\Omega)$, is understood in the sense of distributions,
and one has the relations
\begin{equation}
-\Delta_{min,\Omega}^* = -\Delta_{max,\Omega}, \quad
-\Delta_{min,\Omega} = -\Delta_{max,\Omega}^*
\end{equation}
(see, for instance, \cite[Sect.~3]{BGMM16}).
Invoking the fact that the minimal operator $-\Delta_{min,\Omega}$ is simple (i.e., it has no invariant subspace on which it is self-adjoint), and simple operators have no eigenvalues, $-\Delta_{min,\Omega}$ cannot have any eigenvalues, thus, no nonzero solution $u$ satisfying \eqref{1.5} exists. For recent results of this type see, for instance, \cite[Proposition~2.5]{BR12}. Upon modifications employing appropriate Dirichlet and Neumann traces this approach remains applicable to the more general case of uniformly elliptic second order partial differential operators on Lipschitz domains $\Omega$ (see, e.g., \cite{BGMM16}, \cite{BR12}).

Perhaps, a most illuminating proof of the impossibility of a Dirichlet eigenfunction to be simultaneously a Neumann eigenfunction can be based on the third Green identity, which naturally leads to one of the principal topics of this paper. Assuming again
$\partial \Omega$ to be sufficiently regular (we will treat the case of Lipschitz domains in Section \ref{s5}), we note the following well-known special case of the third Green identity (see, e.g.,
\cite{Co88}, \cite[Theorem~6.10]{Mc00}),
\begin{align}
\begin{split}
u(x) = (\cG_{z} (- \Delta - z) u)(x)
+ (\cD_{z} u)(x) - (\cS_{z} (\partial_\nu u))(x),& \\
u \in H^2(\Omega), \; z \in \bbC, \; x \in \Omega,&  \lb{1.9}
\end{split}
\end{align}
in terms of the resolvent operator $\cG_{z}$, and the single and double layer potentials
$\cS_{z}$ and $\cD_{z}$, $z \in \bbC$, defined by
\begin{align}
& (\cG_{z} f)(x) = \int_{\Omega} E_n^{(0)}(z; x-y) f(y) \, d^n y, \quad f \in L^2(\Omega), \; x \in \Omega,   \lb{1.10} \\
& (\cS_{z} v)(x) = \int_{\partial \Omega} E_n^{(0)} (z; x-\xi) v(\xi) \, d^{n-1} \omega(\xi),
\quad v \in L^2(\partial \Omega), \; x \in \Omega,   \lb{1.11} \\
& (\cD_{z} v)(x) = \int_{\partial \Omega} \big(\partial_{\nu(\xi)} E_n^{(0)}\big) (z; x-\xi) v(\xi) \, d^{n-1} \omega(\xi), \quad v \in L^2(\partial \Omega), \; x \in \Omega.   \lb{1.12}
\end{align}
Here $E_n^{(0)}(z;x)$ represents the fundamental solution of the Helmholtz differential expression $(-\Delta -z)$ in $\bbR^n$, $n\in\bbN$, $n\geq 2$, that is,
\begin{align}
& E_n^{(0)}(z;x) = \begin{cases}
(i/4) \big(2\pi |x|/z^{1/2}\big)^{(2-n)/2} H^{(1)}_{(n-2)/2}
\big(z^{1/2}|x|\big), & n\geq 2,
\; z\in\bbC\backslash \{0\}, \\
\f{-1}{2\pi} \ln(|x|), & n=2, \; z=0, \\
\f{1}{(n-2)\omega_{n-1}}|x|^{2-n}, & n\geq 3, \; z=0,
\end{cases}    \no \\
& \hspace*{6.85cm} \Im\big(z^{1/2}\big)\geq 0,\; x\in\bbR^n\backslash\{0\},   \lb{1.13}
\end{align}
with $H^{(1)}_{\nu}(\, \cdot \,)$ denoting the Hankel function of the first kind
with index $\nu\geq 0$ (cf.\ \cite[Sect.\ 9.1]{AS72}).

Thus, if $u$ is assumed to satisfy \eqref{1.5}, the third Green identity \eqref{1.9} instantly yields
$u \upharpoonright_{\Omega} = 0$ and hence the nonexistence of nontrivial solutions $u$ satisfying \eqref{1.5}. Again, this approach extends to the more general case of uniformly elliptic second order partial differential operators $\cL$ by appropriately replacing the Helmholtz Green's function  $G_n^{(0)} (z; x,y) = E_n^{(0)} (z; x-y)$
of $\cL^{(0)} = - \Delta - z$ in $\bbR^n$ in \eqref{1.9}--\eqref{1.13} by the associated Green's function $G_n(z; x,y)$ of $\cL - z$ in $\bbR^n$.

Although we were interested in properties of Dirichlet eigenfunctions, that is, eigenfunctions of the Dirichlet Laplacian $- \Delta_{D, \Omega}$ in $L^2(\Omega)$, the third Green identity \eqref{1.9} naturally involved the Helmholtz Green's function
$G_n^{(0)} (z; x,y) = E_n^{(0)} (z; x-y)$ for the Laplacian in $L^2(\bbR^n)$. The latter obviously has no knowledge of $\Omega$ and $\partial \Omega$. Moreover, denoting
\begin{equation}
\Omega_+ := \Omega, \quad \Omega_- := \bbR^n \backslash \ol{\Omega},  \lb{1.14}
\end{equation}
with $\Omega_-$ the open exterior of $\Omega$, one is naturally led to a comparison of the Dirichlet Laplacian
\begin{equation}
(- \Delta_{D,\Omega_+}) \oplus (- \Delta_{D, \Omega_-}) \, \text{ in } L^2(\Omega_+) \oplus L^2(\Omega_-) \simeq L^2(\bbR^n)     \lb{1.15}
\end{equation}
and the Laplacian $- \Delta$ in $L^2(\bbR^n)$ (with domain $H^2(\bbR^n)$). While the Dirichlet Laplacian $(- \Delta_{D,\Omega_+}) \oplus (- \Delta_{D, \Omega_-})$ corresponds to a complete decoupling of $\bbR^n$ into $\Omega_+ \cup \Omega_-$ (ignoring the compact boundary $\cC := \partial \Omega_{\pm}$ of $n$-dimensional Lebesgue measure zero), in stark contrast to this decoupling, the Laplacian $- \Delta$ on $H^2(\bbR^n)$ couples $\Omega_+$ and $\Omega_-$ via the imposition of continuity conditions accross $\cC$ of the form
\begin{equation}
u_+\upharpoonright_{\partial \Omega_+} = u_-\upharpoonright_{\partial \Omega_-},
\quad - \partial_{\nu} u_+\upharpoonright_{\partial \Omega_+} = \partial_{\nu} u_-\upharpoonright_{\partial \Omega_-}.    \lb{1.16}
\end{equation}
Here we identified $u \in L^2(\bbR^n)$ with the pair $(u_+, u_-) \in L^2(\Omega_+) \oplus L^2(\Omega_-)$ via $u_{\pm} = u \upharpoonright_{\Omega_{\pm}}$. The relative sign change in the normal derivatives in the second part of \eqref{1.16} is of course being dictated by the opposite orientation of $\nu$ at a point of $\cC = \partial \Omega_{\pm}$. (It should be said that at this point we are purposely a bit cavalier about boundary traces, etc., all this will be developed with complete rigor in the bulk of this paper.) It is this coupling of $\Omega_+$ and
$\Omega_-$ through their joint boundary $\cC$ via the Laplacian $- \Delta$ on
$\bbR^n$ via the continuity requirements \eqref{1.16} that's the second major topic in this paper.

In fact, from this point of view, the open exterior domain $\Omega_-$ is on a similar level as the original domain $\Omega = \Omega_+$ (apart from being unbounded) and introducing the jumps of $u$ and $\partial_{\nu} u$ across $\cC= \partial \Omega_{\pm}$ via
\begin{equation}
[u] := u_+\upharpoonright_{\partial \Omega_+} - u_-\upharpoonright_{\partial \Omega_-}, \quad [\partial_{\nu} u] := - \partial_{\nu} u_+\upharpoonright_{\partial \Omega_+} - \partial_{\nu} u_-\upharpoonright_{\partial \Omega_-},    \lb{1.17}
\end{equation}
the third Green identity \eqref{1.9} can be shown to extend to the following form  (symmetric w.r.t. $\Omega_{\pm}$, cf., e.g., \cite{Co88}, \cite[Theorem~6.10]{Mc00}),
\begin{align}
\begin{split}
u(x) = (\cG_{z} (- \Delta - z) u)(x)
+ (\cD_{z} [u])(x) - (\cS_{z} ([\partial_\nu u]))(x),&   \\
u=(u_+, u_-), \; u_{\pm} \in H^2(\Omega_{\pm}), \; z \in \bbC, \; x \in \bbR^n \backslash \cC.&  \lb{1.18}
\end{split}
\end{align}

At this point we can describe the major objectives of this paper: Decompose a given complex, separable Hilbert space $\sH$ into an orthogonal sum of closed subspaces
$\sH_{\pm}$ as $\sH = \sH_+ \oplus \sH_-$, consider densely defined, closed  symmetric operators $S_{\pm}$ in $\sH_{\pm}$ and their direct sum
$S = S_+ \oplus S_-$ in $\sH$, introduce restrictions $T_{\pm}$ of $S_{\pm}^*$ such that $\ol{T_{\pm}} = S_{\pm}^*$ and appropriate restrictions $A_{0,\pm}$ of $T_{\pm}$, for instance, $A_{0,\pm}$ self-adjoint in $\sH_{\pm}$, defined in terms of certain abstract boundary conditions, and then find a self-adjoint operator $A$ in $\sH$ which closely resembles $A_0 = A_{0,+} \oplus A_{0,-}$, but without any remnants of the boundary conditions in $A_{0,+} \oplus A_{0,-}$ and without any reference to the decomposition of $\sH$ into $\sH_+ \oplus \sH_-$ (i.e., $A$ naturally couples $\sH_{\pm}$ in terms of certain continuity requirements through an abstract ``boundary''). Finally, derive an abstract third Green identity invoking the resolvent (resp., the Green's function $\cG$) of $A$, the operator
$T= T_+ \oplus T_-$, and abstract single and double layer operators constructed from $\cG$. This can indeed be achieved with the help of an appropriate quasi boundary triple for $T$ which also permits one to introduce a natural abstract analog of the ``boundary Hilbert space'' $L^2(\cC)$ in the concrete case of the Laplacian above.

In Section \ref{s2} we briefly recall the basic setup for quasi boundary triples and associated operator-valued Weyl-functions (also called Weyl--Titchmarsh functions) as needed in this paper. The introduction of quasi boundary triples is intimately connected with an abstract (second) Green identity. Section \ref{s3} studies the operator $A$ and derives Krein-type resolvent formulas for it in terms of $A_0$ and a related operator. Section \ref{s4} derives the abstract third Green identity, and finally Section \ref{s5} illustrates the abstract material in Sections \ref{s2}--\ref{s4} with the concrete case of Schr\"odinger operators on Lipschitz domains on smooth, boundaryless Riemannian manifolds.

Finally, we briefly summarize the basic notation used in this paper: Let
$\cH$, $\sH$ be a separable complex Hilbert spaces, $(\cdot,\cdot)_{\cH}$ the
scalar product in $\cH$ (linear in the second factor), and $I_{\cH}$ the identity operator
in $\cH$. If $T$ is a linear operator mapping (a subspace of\,) a
Hilbert space into another, $\dom(T)$ denotes the domain of $T$. The closure
of a closable operator $S$ is denoted by $\ol S$. The spectrum and
resolvent set of a closed linear operator in $\cH$ will be denoted by
$\sigma(\cdot)$  and $\rho(\cdot)$, respectively.
The Banach spaces of bounded linear operators in $\cH$ are
denoted by $\cB(\cH)$; in the context of two
Hilbert spaces, $\cH_j$, $j=1,2$, we use the analogous abbreviation
$\cB(\cH_1, \cH_2)$. The set of all closed linear operators
in ${\cH}$ is denoted by ${\cC}({\cH})$. Moreover,
$\cX_1 \hookrightarrow \cX_2$ denotes the continuous embedding of the Banach space $\cX_1$ into the Banach space $\cX_2$.
We also abbreviate $\bbC_{\pm} := \{z \in \bbC \, | \, \Im(z) \gtrless 0\}$.

\section{Quasi Boundary Triples and their Weyl Functions}  \lb{s2}

In this section we briefly recall the notion of quasi boundary triples and the associated (operator-valued) Weyl functions.

In the following let $S$ be a  densely defined closed symmetric operator in $\sH$.

\begin{definition}\label{Def-triple}
Let $T\subseteq S^*$ be a linear operator in $\sH$ such that $\overline T=S^*$.
A triple $\{\cH,\Gamma_0,\Gamma_1\}$ is called a {\em quasi boundary triple}
for $T$ if $(\cH,(\, \cdot \,, \, \cdot \,)_\cH)$ is a Hilbert space and
$\Gamma_0,\Gamma_1:\dom (T)\rightarrow\cH$
are linear mappings such that the following items $(i)$--$(iii)$ hold:  \\[1mm]
$(i)$ The abstract $($second\,$)$ Green identity
\begin{equation}\label{green1}
      (T f,g)_\sH - (f,Tg)_\sH
      = (\Gamma_1 f,\Gamma_0 g)_\cH
      - (\Gamma_0 f,\Gamma_1 g)_\cH, \quad f,g\in\dom (T),
\end{equation}
is valid. \\[1mm]
$(ii)$ The map $\Gamma = (\Gamma_0,\Gamma_1)^\top : \dom (T) \rightarrow
    \cH^2$ has dense range. \\[1mm]
$(iii)$ $A_0 := T\upharpoonright \ker(\Gamma_0)$ is a self-adjoint operator in $\sH$.
\end{definition}

The notion of quasi boundary triples was introduced in \cite{BL07} and generalizes the concepts of ordinary (and generalized) boundary triples, see, for instance,
\cite{Br76}, \cite{DM95}, \cite{GG91}, \cite{Ko75}, and the references therein. We 
recall that the triple in Definition~\ref{Def-triple} is called an {\it ordinary boundary triple} ({\it generalized boundary triple}) if item $(ii)$ is replaced by the condition
$\ran (\Gamma)=\cH^2$ ($\ran (\Gamma_0)=\cH$, respectively).
On the other hand, the notion of quasi boundary triple is a partial case of the notion of isometric/unitary boundary triples which goes back to Calkin~\cite{Ca39} and
was studied in detail in~\cite{DHMS06}, \cite{DHMS12}.

We recall briefly some important properties of quasi boundary triples. First of all, we note that
a quasi boundary triple for $S^*$ exists if and only if the defect numbers
\begin{equation}
 n_\pm(S)=\dim(\ker(S^*\mp i))
\end{equation}
of $S$ are equal. Next, assume that $\{\cH,\Gamma_0,\Gamma_1\}$ is a quasi boundary triple for $T\subseteq S^*$.
Then the mapping $\Gamma = (\Gamma_0,\Gamma_1)^\top: \dom (T) \to  \cH^2$ is
closable
and $\ker (\Gamma) = \dom (S)$ holds (cf.~\cite[Proposition~2.2]{BL07}). According to \cite[Theorem~2.3]{BL07}
one has
$T=S^*$ if and only if $\ran (\Gamma) = \cH^2$, in this case  the
restriction $A_0=S^*\upharpoonright \ker(\Gamma_0)$ is automatically self-adjoint and the the quasi boundary triple $\{\cH,\Gamma_0,\Gamma_1\}$ is an ordinary boundary triple
in the usual sense. In this context we also note that in the case of finite deficiency indices of $S$ a quasi boundary triple is automatically an ordinary
boundary triple.

Next, the notion of the {\em $\gamma$-field} and {\em Weyl function} associated to a quasi boundary triple will be recalled. The definition is formally the same
as in the case of ordinary and generalized boundary triples. First, one observes that for
each $z \in\rho(A_0)$, the direct sum decomposition
\begin{equation}
  \dom (T) = \dom (A_0)\,\dot +\,\ker(T-z I_{\sH})
  = \ker(\Gamma_0)\,\dot +\,\ker(T-z I_{\sH})
\end{equation}
holds. Hence
the restriction of the mapping $\Gamma_0$ to $\ker(T-z I_{\sH})$
is injective and its range coincides with $\ran(\Gamma_0)$.

\begin{definition}
Let $\{\cH,\Gamma_0,\Gamma_1\}$ be a quasi boundary triple for $T\subseteq S^*$.
The {\em $\gamma$-field} $\gamma$ and the {\em Weyl function} $M$  corresponding to  $\{\cH,\Gamma_0,\Gamma_1\}$ are defined by
\begin{equation}
\rho(A_0) \ni z \mapsto \gamma(z):=\bigl(\Gamma_0\upharpoonright\ker(T-z I_{\sH})\bigr)^{-1},
\end{equation}
and
\begin{equation}
\rho(A_0) \ni z \mapsto M(z) := \Gamma_1 \gamma(z),
\end{equation}
respectively.
\end{definition}

The notions of the  $\gamma$-field and the Weyl function corresponding to
ordinary and generalized boundary triples were introduced in~\cite{DM91} and~\cite{DM95}, respectively.
In both cases the Weyl function $M$ turns out to be a Herglotz--Nevanlinna function with values in $\cB(\cH)$, that is, $M$ is holomorphic on $\bbC\backslash\bbR$, and
\begin{equation}\label{jaok}
\Im(z) \Im (M(z)) \ge 0 \, \text{ and } \, M(z)=M(\ol z)^*, \quad z\in\bbC\backslash\bbR.
\end{equation}
The values of the $\gamma$-field are bounded operators from $\cH$ into $\sH$ with $\ran(\gamma(z))=\ker(T-z I_{\sH})$
and the following identity holds
\begin{equation}\label{jaok2}
M(z)-M(z)^*=(z-\bar z)\gamma(z)^*\gamma(z),\quad  z\in\rho(A_0).
\end{equation}
In the case of a quasi boundary triple the operators $\gamma (z)$, $z\in\rho(A_0)$, are defined on the dense subspace $\ran(\Gamma_0)\subseteq \cH$ and map onto
$\ker(T-z I_{\sH}) \subset\sH$. By \cite[Proposition~2.6]{BL07} the operator $\gamma(z)$ is bounded and hence admits a continuous extension onto $\cH$.
Furthermore, one has
\begin{equation}\label{gsgs}
  \gamma(z)^*:\sH\rightarrow\cH,\quad f\mapsto\gamma(z)^*f
  = \Gamma_1(A_0-\ol{z} I_{\sH})^{-1}f,\quad z\in\rho(A_0).
\end{equation}
The values of the Weyl function $M(z)$, $z \in \rho(A_0)$  are operators in $\cH$ defined on
$\ran(\Gamma_0)$ and mapping into $\ran(\Gamma_1)$.
The analogs of \eqref{jaok} and \eqref{jaok2}, and various other useful and important properties of the Weyl function
can be found in \cite{BL07}, \cite{BL12}. In particular,
\begin{equation}
M(z)\subseteq M(\ol{z})^*, \quad z \in \rho(A_0),
\end{equation}
and hence
the operators $M(z)$, $z \in \rho(A_0)$, are closable operators in $\cH$.
We point out that the operators $M(z)$, $z \in \rho(A_0)$, and their closures
are generally unbounded.

\section{The Coupling Model}\lb{s3}

In this section we discuss the coupling issue mentioned in \eqref{1.14}--\eqref{1.16} from a purely abstract point of view.

Let $S_+$ and $S_-$ be densely defined closed symmetric operators in the separable Hilbert spaces $\sH_+$ and $\sH_-$, respectively, and assume that
the defect indices of $S_+$ and $S_-$ satisfy
\begin{equation}
 n_+(S_+)=n_-(S_+)=n_+(S_-)=n_-(S_-)=\infty.
\end{equation}
The case of finite defect numbers can be treated with the help of ordinary boundary triples in an efficient way and will not be discussed
here (cf.\ \cite{DHMS00}).

Let $T_+$ and $T_-$ be such that $\overline T_+=S_+^*$ and $\overline T_-=S_-^*$, and assume that $\{\cH,\Gamma_0^+,\Gamma_1^+\}$
and  $\{\cH,\Gamma_0^-,\Gamma_1^-\}$ are quasi boundary triples for $S_+^*$ and $S_-^*$, respectively. The corresponding $\gamma$-fields
and Weyl functions are denoted by $\gamma_+$ and $\gamma_-$, and $M_+$ and $M_-$, respectively. Furthermore, let
\begin{equation}
A_{0,+}=T_+\upharpoonright\ker(\Gamma_0^+), \quad A_{0,-}=T_-\upharpoonright\ker(\Gamma_0^-).
\end{equation}
It is important to note that the identities
\begin{equation}\label{gamstar}
 \gamma_+(z)^*=\Gamma_1^+(A_{0,+}-\ol{z})^{-1}, \quad
 \gamma_-(z)^*=\Gamma_1^-(A_{0,-}-\ol{z})^{-1}, 
\end{equation}
hold for all $z \in \rho(A_{0,+})$ and $z \in \rho(A_{0,-})$, respectively (cf.\ \eqref{gsgs}).
In the following consider the operators
\begin{equation}\label{sts}
 S=\begin{pmatrix} S_+ & 0 \\ 0 & S_-\end{pmatrix}, \quad T=\begin{pmatrix} T_+ & 0 \\ 0 & T_-\end{pmatrix}, \quad S^*=\begin{pmatrix} S_+^* & 0 \\ 0 & S_-^*\end{pmatrix},
\end{equation}
in the Hilbert space $\sH=\sH_+\oplus\sH_-$. It is clear that $S$ is a closed, densely defined, 
symmetric operator in $\sH$ with equal infinite defect numbers,
\begin{equation}
n_+(S)=n_-(S)=\infty,
\end{equation}
and that
\begin{equation}
\ol{T} = S^*. 
\end{equation}
The elements $f$ in the domain of $S$, $T$ and
$S^*$ will be written as two component vectors of the form $f=(f_+,f_-)^\top$, where
$f_\pm$ belongs to the domain of $S_\pm$, $T_\pm$ and $S_\pm^*$, respectively. It is easy to see that $\{\cH\oplus\cH,\Gamma_0,\Gamma_1\}$, where
\begin{equation}\label{qbtcoup}
\Gamma_0f=\begin{pmatrix} \Gamma_0^+f_+\\ \Gamma_0^-f_-\end{pmatrix}, \quad\Gamma_1f=\begin{pmatrix} \Gamma_1^+f_+\\ \Gamma_1^-f_-\end{pmatrix},\quad f\in\dom (T),
\end{equation}
is a quasi boundary triple for $S^*$ such that
\begin{equation}\label{aaa}
 A_0=T\upharpoonright\ker(\Gamma_0)=\begin{pmatrix} T_+\upharpoonright\ker(\Gamma_0^+) & 0 \\ 0 & T_-\upharpoonright\ker(\Gamma_0^-)\end{pmatrix}
 =\begin{pmatrix} A_{0,+} & 0 \\ 0 & A_{0,-}\end{pmatrix}.
\end{equation}
One notes that $\rho(A_0)=\rho(A_{0,+})\cap\rho(A_{0,-})$.
The $\gamma$-field $\gamma$ and Weyl function $M$ corresponding to the quasi boundary triple $\{\cH\oplus\cH,\Gamma_0,\Gamma_1\}$ are given by
\begin{align}\label{gamgam}
 \gamma(z)&=\begin{pmatrix} \gamma_+(z) & 0 \\ 0 & \gamma_-(z) \end{pmatrix},
 \quad z \in \rho(A_0),   \\
 M(z)&=\begin{pmatrix} M_+(z) & 0 \\ 0 & M_-(z) \end{pmatrix}, \quad
 z \in \rho(A_0),
\end{align}
and \eqref{gamstar} implies
\begin{equation}\label{gamgamstar}
 \gamma(z)^*f=\begin{pmatrix}
               \Gamma_1^+(A_{0,+}-\ol{z})^{-1} f_+ \\ \Gamma_1^-(A_{0,-}-\ol{z})^{-1}f_-
              \end{pmatrix},\quad f=(f_+,f_-)^\top\in\sH.
\end{equation}

The next result, Theorem \ref{resthm}, can be viewed as an abstract analogue of the coupling of differential operators, where $\Gamma_0^\pm$ are Dirichlet trace operators
and $\Gamma_1^\pm$ are Neumann trace operators acting on different domains
(cf.\ Section~\ref{s5} for more details). We also note that in the following the operators
$M_+(z)+M_-(z)$ in $\cH$ are assumed to be defined on
$\dom(M_+(z)) \cap \dom(M_-(z))$.

\begin{theorem}\label{resthm}
 Let $\{\cH,\Gamma_0^+,\Gamma_1^+\}$ and  $\{\cH,\Gamma_0^-,\Gamma_1^-\}$ be quasi boundary triples for $S_+^*$ and $S_-^*$ with Weyl functions $M_\pm$ and define
 \begin{equation}
  A:= T\upharpoonright\bigl\{f=(f_+,f_-)^\top\in\dom (T) \,\big| \, \Gamma_0^+f_+=\Gamma_0^-f_-,\,\Gamma_1^+f_+=-\Gamma_1^-f_- \bigr\}.
 \end{equation}
Then the following assertions $(i)$--$(iii)$ hold: \\[1mm]
$(i)$ $A$ is a symmetric operator in $\sH$. \\[1mm]
$(ii)$ $z \in \rho(A_0)$ is an eigenvalue of $A$ if and only if
 \begin{equation}
  \ker\bigl(M_+(z)+M_-(z)\bigr)\not=\{0\}.
  \end{equation}
$(iii)$ $A$ is a self-adjoint operator in $\sH$ if and only if
 \begin{equation}
  \ran\bigl(\Gamma^\pm_1\upharpoonright\dom (A_{0,\pm}) \bigr)
  \subset \ran\bigl(M_+(z)+M_-(z)\bigr)
 \end{equation}
 holds for some $($and hence for all\,$)$ $z \in \dC_+$ and some
 $($and hence for all\,$)$ $z \in \dC_-$.
\\[2mm]
If $A$ is a self-adjoint operator in $\sH$ then for all $z \in \rho(A)\cap\rho(A_0)$ the resolvent of $ A$ is given in terms of a Krein-type resolvent formula by
\begin{equation}\label{r1}
 (A-z I_{\sH})^{-1}=(A_0-z I_{\sH})^{-1}+ \gamma(z)\Theta(z)
\gamma(\ol{z})^*,
\end{equation}
where
\begin{equation}\label{thetayes}
 \Theta(z) = - \begin{pmatrix} (M_+(z)+M_-(z))^{-1} & (M_+(z)+M_-(z))^{-1} \\
(M_+(z)+M_-(z))^{-1} & (M_+(z)+M_-(z))^{-1} \end{pmatrix}.
\end{equation}
\end{theorem}

\begin{remark}
One notes that the perturbation term $\gamma(z)\Theta(z)\gamma(\ol{z})^*$ on the
right-hand side of \eqref{r1} can also be written in the form
\begin{equation}
- \begin{pmatrix}  \gamma_+(z) \\ \gamma_-(z)
 \end{pmatrix} (M_+(z)+M_-(z))^{-1}
\bigl(\gamma_+(\ol{z})^*,\,\gamma_-(\ol{z})^*\bigr).
\end{equation} 
\end{remark}

\begin{proof}[Proof of Theorem~\ref{resthm}]
$(i)$ In order to show that $A$ is a symmetric operator in $\sH$  let $f=(f_+,f_-)^\top$ and  $g=(g_+,g_-)^\top$ be in $\dom (A)$.
Making use of the abstract boundary conditions for $f,g\in\dom(A)$, a straightforward computation using the abstract Green identity \eqref{green1} shows that
\begin{align}
 &(Af,g)_\sH-(f,Ag)_\sH  \no \\
 &\quad =(T_+f_+,g_+)_{\sH_+}-(f_+,T_+g_+)_{\sH_+}+(T_-f_-,g_-)_{\sH_-}-(f_-,T_-g_-)_{\sH_-}   \no \\
 &\quad =(\Gamma_1^+f_+,\Gamma_0^+g_+)_{\cH}
 - (\Gamma_0^+f_+,\Gamma_1^+g_+)_{\cH}
 + (\Gamma_1^-f_-,\Gamma_0^-g_-)_{\cH}
 - (\Gamma_0^-f_-,\Gamma_1^-g_-)_{\cH}
 \no \\
 &\quad =(\Gamma_1^+f_++\Gamma_1^-f_-,\Gamma_0^+g_+)_{\cH}
 -(\Gamma_0^+f_+,\Gamma_1^+g_++\Gamma_1^-g_-)_{\cH}    \no \\
 &\quad =0,
\end{align}
hence $A$ is symmetric. \\[2mm]
$(ii)$ Let $z\in\rho(A_0)$ and assume that $z$ is an eigenvalue of $A$. Considering $f\in\ker(A-z I_{\sH})$, $f\not=0$,
one observes that $\Gamma_0f\not=0$ as otherwise $f\in\dom (A_0)$ would be an eigenfunction of $A_0$ at $z$. Clearly $f\in\ker(T-z I_{\sH})$ and hence
$M(z)\Gamma_0 f=\Gamma_1 f$. For $f=(f_+,f_-)^\top$ one then obtains 
\begin{equation}
M_+(z)\Gamma_0^+f_+ =\Gamma^+_1 f_+ \, \text{ and } \, M_-(z) \Gamma_0^- f_-=\Gamma_1^- f_-.
\end{equation}
Since $f\in\dom(A)$ satisfies $\Gamma_0^+f_+=\Gamma_0^-f_-$ and $\Gamma_1^+f_++\Gamma_1^-f_-=0$, one concludes
\begin{equation}
\bigl(M_+(z)+M_-(z)\bigr)\Gamma_0^\pm f_\pm=M_+(z)\Gamma_0^+f_++M_-(z)\Gamma_0^-f_-=\Gamma_1^+f_++\Gamma_1^-f_-=0.
\end{equation}
As a consequence of $\Gamma_0^\pm f_\pm\not=0$, it follows that $\ker(M_+(z)+M_-(z))\not=\{0\}$.

Conversely, assume that $\varphi\in\ker(M_+(z)+M_-(z))$, $\varphi\not=0$, for some $z\in\rho(A_0)$. One notes that
\begin{equation}\label{mm}
M_+(z)\varphi=-M_-(z)\varphi \, \text{ and } \, \varphi\in\ran (\Gamma_0^+)\cap\ran(\Gamma_0^-).
\end{equation}
Hence there exist $f_+\in\ker(T_+-z I_{\sH_+})$ and
$f_-\in\ker(T_- - z I_{\sH_-})$ such that
\begin{equation}\label{bc1}
\Gamma_0^+f_+=\varphi=\Gamma_0^-f_-.
\end{equation}
From the definition of $M_+$ and $M_-$, and \eqref{mm}--\eqref{bc1} one concludes
that
\begin{equation}\label{bc5}
 \Gamma_1^+f_+=M_+(z)\Gamma_0^+f_+=-M_-(z)\Gamma_0^-f_-=-\Gamma_1^-f_-,
\end{equation}
and hence \eqref{bc1} and \eqref{bc5} show that
$f=(f_+,f_-)^\top\in\ker(T-z I_{\sH})$ satisfies both abstract boundary conditions for elements in $\dom(A)$.
Thus, $f\in\ker(A-z I_{\sH})$. \\[2mm]
$(iii)$ First, assume that $A$ is a self-adjoint operator in $\sH$, fix $z\in\dC\backslash\dR$, and let
\begin{equation}\label{rcm1}
 \varphi_+\in\ran\bigl(\Gamma^+_1\upharpoonright\dom (A_{0,+})\bigr).
\end{equation}
Then there exists $f_+\in\sH_+$ such that
\begin{equation}\label{g++}
 \varphi_+=\Gamma^+_1(A_{0,+}- z I_{\sH_+})^{-1}f_+=\gamma_+(\overline{z})^*f_+,
\end{equation}
where the last identity follows from \eqref{gamstar}.
Next, consider $f=(f_+,0)^\top\in\sH$, set
\begin{equation}
 h:=(A- z I_{\sH})^{-1}f-(A_0- z I_{\sH})^{-1} f\in\ker(T- z I_{\sH}), 
\end{equation}
and
\begin{equation}
 k:=(A- z I_{\sH})^{-1}f\in\dom(A).
\end{equation}
Then one has $\Gamma_0 h=\Gamma_0 k$ and $\Gamma_1 h
= \Gamma_1 k - \gamma(\overline{z})^*f$ and hence
\begin{equation}
 \gamma(\overline{z})^*f=\Gamma_1 k-\Gamma_1 h 
 = \Gamma_1 k-M(z)\Gamma_0 h=\Gamma_1 k-M(z)\Gamma_0 k.
\end{equation}
Making use of \eqref{g++} and $f_-=0$ this reads componentwise as
\begin{equation}
 \begin{split}
   \varphi_+=\gamma_+(\overline{z})^*f_+&=\Gamma^+_1 k_+-M_+(z)\Gamma^+_0 k_+,\\
    0=\gamma_-(\overline{z})^*f_-&=\Gamma^-_1 k_--M_-(z)\Gamma^-_0 k_- 
\end{split}
\end{equation}
(cf.\ \eqref{gamgamstar}).
Summing up these two equations and taking into account that  $k\in\dom(A)$ satisfies 
$\Gamma_1^+k_++\Gamma_1^-k_-=0$ and $\Gamma_0^+k_+=\Gamma_0^-k_-$, one finds
\begin{equation}
 \varphi_+=\gamma_+(\overline{z})^*f_+ = - \bigl(M_+(z)+M_-(z) \bigr)\Gamma^+_0 k_+.
\end{equation}
Hence, the inclusion
\begin{equation}
\ran\bigl(\Gamma^+_1\upharpoonright\dom (A_{0,+}) \bigr)
\subseteq \ran\bigl(M_+(z)+M_-(z)\bigr)
\end{equation}
holds for any $z \in\dC\backslash\dR$. In the same way as above one also shows the inclusion
\begin{equation}
\ran\bigl(\Gamma^-_1\upharpoonright\dom (A_{0,-}) \bigr)\subseteq
\ran\bigl(M_+(z)+M_-(z)\bigr).
\end{equation}

Next, we will prove the converse. Assume that
 \begin{equation}
  \ran\bigl(\Gamma^\pm_1\upharpoonright\dom (A_{0,\pm}) \bigr)\subseteq
  \ran\bigl(M_+(z)+M_-(z)\bigr)
 \end{equation}
holds for some $z\in\dC_+$ and some $z\in\dC_-$. We have to prove that the operator $A$ is self-adjoint in $\sH$. Along the way we will also show that the resolvent formula
holds at the point $z$. Note first that $A$ is symmetric by item $(i)$ and hence all eigenvalues of $A$ are real. In particular, $z$ is not an eigenvalue of $A$
and according to item $(ii)$, the operator $M_+(z)+M_-(z)$ is injective. Let $f=(f_+,f_-)^\top\in\sH$ and note that (cf.\ \eqref{gamgamstar})
\begin{equation}
 \gamma(\overline{z})^*f=\begin{pmatrix} \Gamma_1^+(A_{0,+}- z I_{\sH_+})^{-1}f_+ \\
 \Gamma_1^-(A_{0,-}- z I_{\sH_-})^{-1}f_- \end{pmatrix}, 
\end{equation}
and that
\begin{equation}
 \Gamma_1^\pm(A_{0,\pm}- z I_{\sH_{\pm}})^{-1}f_\pm \in
 \ran\bigl(M_+(z)+M_-(z)\bigr) = \dom\big((M_+(z)+M_-(z))^{-1}\big), 
\end{equation}
by assumption. Now consider the element
\begin{equation}
 g:= (A_0- z I_{\sH})^{-1}f - \gamma(z)
 \begin{pmatrix} (M_+(z)+M_-(z))^{-1} & (M_+(z)+M_-(z))^{-1} \\ 
 (M_+(z)+M_-(z))^{-1} & (M_+(z)+M_-(z))^{-1} \end{pmatrix}
\gamma(\overline{z})^*f, 
\end{equation}
which is well-defined by the above considerations and the fact that
\begin{equation}
 \dom (\gamma(z)) = \dom\left(\begin{pmatrix} \gamma_+(z) & 0 \\ 
 0 & \gamma_-(z) \end{pmatrix}\right)=\ran (\Gamma_0^+) \times \ran(\Gamma_0^-), 
\end{equation}
and
\begin{equation}
\ran\big((M_+(z)+M_-(z))^{-1}\big) = \dom \bigl(M_+(z)+M_-(z)\bigr)
=\ran(\Gamma_0^+)\cap\ran(\Gamma_0^-).
\end{equation}
Since $(A_0- z I_{\sH})^{-1}f\in\dom (A_0)\subset \dom (T)$ and
$\ran(\gamma(z))=\ker(T- z I_{\sH})\subset\dom (T)$, it is clear that $g \in \dom (T)$. 
Next, it will be shown that
$g=(g_+,g_-)^\top$ satisfies the boundary conditions
\begin{equation}\label{bc2}
 \Gamma_0^+g_+=\Gamma_0^-g_- \, \text{ and } \, \Gamma_1^+g_+=-\Gamma_1^-g_-.
\end{equation}
Due to
\begin{equation}
 (A_0- z I_{\sH})^{-1}=\begin{pmatrix} (A_{0,+}- z I_{\sH_+})^{-1} & 0 \\
 0 &  (A_{0,-}- z I_{\sH_-})^{-1}\end{pmatrix},
\end{equation}
and the special form of $\gamma(z)$ and $\gamma(\overline{z})^*$, one infers that
\begin{equation}
 g_+= (A_{0,+}- z I_{\sH_+})^{-1}f_+  - \gamma_+(z) \big(M_+(z)+M_-(z)\big)^{-1}
 \big(\gamma_+(\overline{z})^*f_+ + \gamma_-(\overline{z})^*f_-\big),
\end{equation}
and
\begin{equation}
 g_-= (A_{0,-}- z I_{\sH_-})^{-1}f_-  - \gamma_-(z) \big(M_+(z)+M_-(z)\big)^{-1}
 \big(\gamma_+(\overline{z})^*f_+ + \gamma_-(\overline{z})^*f_-\big).
\end{equation}
Since by definition, $A_{0,\pm}=T_\pm\upharpoonright\ker(\Gamma_0^\pm)$ and
$\gamma_\pm(z)=(\Gamma_0^\pm\upharpoonright\ker(T_\pm- z I_{\sH_{\pm}}))^{-1}$, one obtains
\begin{align}
 \Gamma_0^+g_+ &= - \big(M_+(z)+M_-(z)\big)^{-1}
 \big(\gamma_+(\overline{z})^*f_+
 + \gamma_-(\overline{z})^*f_-\big),   \\
 \Gamma_0^-g_- &= - \big(M_+(z)+M_-(z)\big)^{-1}
 \big(\gamma_+(\overline{z})^*f_+
 +\gamma_-(\overline{z})^*f_-\big),
\end{align}
and hence the first condition in \eqref{bc2} is satisfied. Next, we make use of \eqref{gamstar} 
and $M_\pm(z)=\Gamma_1^\pm\gamma_\pm(z)$ and compute
\begin{align}
 \Gamma_1^+g_+ &=\gamma_+(\overline{z})^*f_+ - M_+(z) \big(M_+(z)+M_-(z)\big)^{-1}
                    \big(\gamma_+(\overline{z})^*f_+ + \gamma_-(\overline{z})^*f_-\big),  \\
 \Gamma_1^-g_- &=\gamma_-(\overline{z})^*f_- - M_-(z) \big(M_+(z)+M_-(z)\big)^{-1}
                    \big(\gamma_+(\overline{z})^*f_+ + \gamma_-(\overline{z})^*f_-\big).
\end{align}
It follows that $\Gamma_1^+g_+ + \Gamma_1^-g_-=0$ and hence also the second boundary 
condition in \eqref{bc2} is satisfied. Therefore, $g \in \dom(A)$, and when applying $(A- z I_{\sH})$ 
to $g$ it follows from the particular form of $g$ and $\ran(\gamma(z))\subseteq \ker(T- z I_{\sH})$ that
\begin{equation}
 (A- z I_{\sH}) g=(T- z I_{\sH})(A_0- z I_{\sH})^{-1}f=f.
\end{equation}
Furthermore, as $A$ is symmetric, $z$ is not an eigenvalue of $A$ and one concludes that
\begin{equation}
 (A- z I_{\sH})^{-1}f=g.
\end{equation}
Since $f\in\sH$ was chosen arbitrary it follows that $(A- z I_{\sH})^{-1}$
is an everywhere defined operator in $\sH$.
By our assumptions this
is true for a point $z\in\dC_+$ and for a point $z \in\dC_-$. Hence it follows that $A$ is self-adjoint and that the resolvent of $A$
at the point $z$ has the asserted form, proving assertion $(iii)$.
\\[2mm]
It remains to show that the resolvent of $A$ is of the form as stated in the theorem for all $z \in\rho(A)\cap\rho(A_0)$. For this we remark that
\begin{equation}\label{rcm}
 \gamma_+(\overline{z})^*f_+ + \gamma_-(\overline{z})^*f_-
 \in \ran\bigl(M_+(z)+M_-(z)\bigr)
\end{equation}
holds for all $z\in\rho(A)\cap\rho(A_0)$ and $f=(f_+,f_-)^\top\in\sH$; this follows essentially from
the first part of the proof of item $(iii)$ (which remains valid for points in $\rho(A)\cap\rho(A_0)$).
Based on \eqref{rcm} and the fact that $\ker(M_+(z)+M_-(z))=\{0\}$ for all
$z \in\rho(A)\cap\rho(A_0)$ by assertion $(ii)$, it can be shown in the same way as in the second part of the proof of item $(iii)$ that the resolvent of $A$ has the asserted form. This completes the proof of Theorem~\ref{resthm}.
\end{proof}

The next step is to derive a slightly modified formula for the resolvent of $A$ in Theorem~\ref{resthm}
where the resolvent of $A_0$ is replaced by the resolvent of the operator
\begin{equation}
 \begin{pmatrix} A_{0,+} & 0 \\ 0 & A_{1,-} \end{pmatrix},
\end{equation}
where
\begin{equation}
A_{1,-}=T_-\upharpoonright\ker(\Gamma_1^-)
\end{equation}
is assumed to be  a self-adjoint operator in $\sH_-$. We recall that in the context of quasi boundary triples, the extension $A_{1,-}$ of $S_-$ corresponding to
$\ker(\Gamma_1^-)$ is always symmetric, but generally not self-adjoint.
The resolvent formula in the next theorem is essentially a consequence of Theorem~\ref{resthm} and a formula relating the resolvent
of $A_{0,-}$ with the resolvent of $A_{1,-}$.

\begin{theorem}\label{resthm2}
Let $\{\cH,\Gamma_0^+,\Gamma_1^+\}$ and  $\{\cH,\Gamma_0^-,\Gamma_1^-\}$ be quasi boundary triples for $S_+^*$ and $S_-^*$ with Weyl functions $M_\pm$
as in Theorem~\ref{resthm}.
Assume, in addition, that $A_{1,-}=T_-\upharpoonright\ker(\Gamma_1^-)$ is self-adjoint in $\sH_-$, 
and let $A$ in Theorem~\ref{resthm} be self-adjoint in $\sH$.
Then for all $z \in\rho(A)\cap\rho(A_{0,+})\cap\rho(A_{0,-})\cap\rho(A_{1,-})$ the resolvent of $A$ is given by
\begin{equation}\label{AA+-}
 (A- z I_{\sH})^{-1}=\left(\begin{pmatrix} A_{0,+} & 0 \\ 0 & A_{1,-} \end{pmatrix}
 - z I_{\sH}\right)^{-1}
 +\widehat \gamma(z) \Sigma(z) \widehat \gamma(\overline{z})^*,
\end{equation}
where
\begin{equation}
 \widehat\gamma(z)=\begin{pmatrix} \gamma_+(z) & 0 \\
 0 & \gamma_-(z)M_-(z)^{-1}\end{pmatrix},
 \quad \Sigma(z) =- \begin{pmatrix} M_+(z) & I_{\cH} \\
 I_{\cH} & -M_-(z)^{-1} \end{pmatrix}^{-1}. 
\end{equation}
\end{theorem}
\begin{proof}
Since $A_{1,-}=T_-\upharpoonright\ker(\Gamma_1^-)$ is self-adjoint in $\sH_-$ it follows from \cite[Theorem~6.16]{BL12}
that $M_-(z)$ is injective for all $z \in\rho(A_{0,-})\cap\rho(A_{1,-})$ and the resolvents of $A_{0,-}$ and $A_{1,-}$ are related via
\begin{equation}\label{res01}
 (A_{0,-}- z I_{\sH_-})^{-1}=(A_{1,-}- z I_{\sH_-})^{-1}+\gamma_-(z)
 M_-(z)^{-1}\gamma_-(\overline{z})^*
\end{equation}
for all $z \in\rho(A_{0,-})\cap\rho(A_{1,-})$.
Making use of \eqref{aaa} and inserting \eqref{res01} in \eqref{r1}--\eqref{thetayes} one obtains
\begin{equation}
 (A- z I_{\sH})^{-1}=\left(\begin{pmatrix} A_{0,+} & 0 \\ 0 & A_{1,-} \end{pmatrix}
 - z I_{\sH}\right)^{-1}
+ \widehat\gamma(z) \Psi(z)\widehat \gamma(\ol{z})^*,
\end{equation}
where
\begin{equation}
\Psi(z)=\begin{pmatrix} I_{\cH} & 0 \\ 0 & M_-(z)\end{pmatrix}
\left[\begin{pmatrix} 0 & 0 \\ 0 &  M_-(z)^{-1}\end{pmatrix}+\Theta(z) \right]
\begin{pmatrix} I_{\cH} & 0 \\ 0 & M_-(z)\end{pmatrix}
\end{equation}
and we have used that
\begin{equation}
 \bigl(\gamma_-(\overline{z})M_-(\overline{z})^{-1}\bigr)^*
 = \bigl(M_-(\overline{z})^{*}\bigr)^{-1}\gamma_-(\overline{z})^*
 =M_-(z)^{-1}\gamma_-(\overline{z})^*.
\end{equation}
Here the first equality holds since $\gamma_-(\overline{z})^*$ is everywhere defined and bounded, and the second equality is valid
since $M_-(z)^{-1}\subset (M_-(\overline{z})^*)^{-1}$ and
\begin{equation}
\ran(\gamma_-(\overline{z})^*)\subset\ran(\Gamma_1)
= \dom \big(M_-(z)^{-1}\big), \quad z \in\rho(A_{0,-})\cap\rho(A_{1,-}). 
\end{equation}
The block operator matrix $\Psi(z)$ in $\cH^2$
has the form
\begin{equation}
\Psi(z)=\begin{pmatrix} - \big(M_+(z) + M_-(z)\big)^{-1} & - \big(M_+(z) + M_-(z)\big)^{-1} M_-(z) \\
- M_-(z) \big(M_+(z) + M_-(z)\big)^{-1} & M_+(z) \big(M_+(z) + M_-(z)\big)^{-1} M_-(z) \end{pmatrix}, 
\end{equation}
and from this representation one infers that $\Psi(z)=\Sigma(z)$ for all
$z \in\rho(A)\cap\rho(A_{0,+})\cap\rho(A_{0,-})\cap\rho(A_{1,-})$.
\end{proof}

In the case where $\{\cH,\Gamma_0^+,\Gamma_1^+\}$ is an ordinary boundary triple, equation \eqref{AA+-} was proved in \cite[eq.~(6.7)]{DHMS09} (cf.\ also \cite{DHMS00}).

\section{The Third Green Identity}  \lb{s4}

This section is devoted to an abstract version of the Third Green identity (cf.\
\eqref{1.18} for the concrete example that motivated these investigations).

We will investigate the operator
\begin{equation}\label{aaaa}
  A= T\upharpoonright \big\{f=(f_+,f_-)^\top\in\dom (T) \, \big| \, \Gamma_0^+f_+=\Gamma_0^-f_-,\,\Gamma_1^+f_+=-\Gamma_1^-f_- \big\},
 \end{equation}
which corresponds to the coupling of the quasi-boundary triples
$\{\cH,\Gamma_0^+,\Gamma_1^+\}$ and $\{\cH,\Gamma_0^-,\Gamma_1^-\}$ for $S_-^*$ and $S_+^*$ in Theorem~\ref{resthm}.
One recalls that $A$ is a symmetric operator in the Hilbert space $\sH=\sH_+\oplus\sH_-$. From now on we shall assume that the following hypothesis is satisfied
(cf.\ Theorem~\ref{resthm}\,$(iii)$).

\begin{hypothesis}\label{h1}
The operator $A$ in \eqref{aaaa} is self-adjoint in the Hilbert space $\sH$.
\end{hypothesis}

In the following let $\sH_2:=\dom (A)$ be the Hilbert space with inner product
\begin{equation}\label{gn}
 (f,g)_{\sH_2}:=(A f,Ag)_\sH + (f,g)_\sH,\quad f,g\in\dom(A),
\end{equation}
and let $\sH_{-2}$ be the adjoint space of distributions on $\sH_2$ with the pairing
denoted by
\begin{equation}
{}_{\sH_{-2}}\langle h ,g\rangle_{\sH_2}, \quad g\in\sH_2, \;  h\in{\sH}_{-2}.
\end{equation}
Since for every $f\in\sH$ the functional $(f, \, \cdot \,)_\sH$ is bounded on $\sH_2$, the space $\sH$ embeds densely into the space $\sH_{-2}$ in such a way that
(cf.\ \cite[Section~1]{Be86})
\begin{equation}
{}_{\sH_{-2}}\langle f,g\rangle_{\sH_2} = (f,g)_\sH, \quad f\in{\sH}, \; g\in\sH_2,
\end{equation}
leading to a Gelfand triple of the form
\begin{equation}\label{gelfand}
\sH_2 \hookrightarrow \sH \hookrightarrow \sH_{-2}.
\end{equation}
Let $\cA$ be the dual operator to $A$ determined by
\begin{equation}\label{eq:cA}
{}_{\sH_{-2}}\langle \cA f,g\rangle_{\sH_2}:=(f,Ag)_{\sH}, \quad f\in\sH, \; g\in\sH_2.
\end{equation}
Since $A \in \cB(\sH_2, \sH)$, also $\cA \in \cB(\sH, \sH_{-2})$.

Next, define the map $\Upsilon$ on $\sH_2=\dom(A)$ as the restriction of 
$(\Gamma_0^+,\Gamma_1^+)^\top$,
\begin{equation}\label{gammagamma}
 \Upsilon:\sH_2\to\cH^2,\quad f\mapsto \Upsilon f= \begin{pmatrix} \Upsilon_0 f \\
 \Upsilon_1 f\end{pmatrix}:=\begin{pmatrix} \Gamma_0^+ f_+\\
 \Gamma_1^+ f_+\end{pmatrix}=\begin{pmatrix} \Gamma_0^- f_-\\
 -\Gamma_1^- f_-\end{pmatrix}.
\end{equation}
Here the last equality follows from the abstract boundary conditions in \eqref{aaaa} for all 
$f \in \dom(A)$.

\begin{lemma}\label{p3.4}
Assume Hypothesis~\ref{h1} and suppose that $A$ in \eqref{aaaa} corresponds to the coupling of
the quasi-boundary triples $\{\cH,\Gamma_0^+,\Gamma_1^+\}$ and
$\{\cH,\Gamma_0^-,\Gamma_1^-\}$ for $S_-^*$ and $S_+^*$, respectively.
Then the operators $\Upsilon_0,\Upsilon_1:\sH_2\to\cH$ in \eqref{gammagamma} are bounded.
\end{lemma}
\begin{proof}
We start by showing that the map $\Upsilon$ is closable from $\sH_2$ to $\cH^2$.
Assume that
\begin{equation}
\lim_{n\to\infty} f_n = 0,\quad \lim_{n\to\infty}A f_n = 0,\quad
\lim_{n\to\infty}\Upsilon_0f_n = \varphi,\quad \lim_{n\to\infty} \Upsilon_1f_n = \psi
\end{equation}
for some $\varphi,\psi\in\cH$.
Then by \eqref{gammagamma},
\begin{equation}
\lim_{n\to\infty} f_{n,+} = 0, \quad \lim_{n\to\infty} T_+ f_{n,+} = 0, \quad
\lim_{n\to\infty} \Gamma_0^+ f_{n,+} = \varphi, \quad \lim_{n\to\infty} \Gamma_1^+f_{n,+} = \psi.
\end{equation}
Since the map $(\Gamma_0^+,\Gamma_1^+)^\top:\dom(T)\rightarrow\cH^2$ is closable by \cite[Proposition 2.2]{BL07} one concludes that
$\varphi=\psi=0$ and hence the map $\Upsilon:\sH_2\to\cH^2$
is closable. Since $\dom(\Upsilon)=\sH_2$, it follows that $\Upsilon$ is closed, and the closed graph theorem
implies that $\Upsilon:\sH_2\to\cH^2$ is bounded. Thus, also $\Upsilon_0,\Upsilon_1:\sH_2\to\cH$
are bounded.
\end{proof}

Let $\Upsilon_j^*:\cH\to \sH_{-2}$ be the dual operator to $\Upsilon_j:\sH_2\to\cH$,
$j=0,1$, in \eqref{gammagamma} determined by
\begin{equation}\label{eq:Gamma_pm}
{}_{\sH_{-2}}\langle \Upsilon_j^* \varphi ,g\rangle_{\sH_2}:=(\varphi ,\Upsilon_j g)_{\cH}, \quad \varphi \in\cH, \; g\in\sH_2, \; j=0,1.
\end{equation}
Since $\Upsilon_j$ are bounded operators from $\sH_2$ to $\cH$ by Lemma~\ref{p3.4}, it is clear that
$\Upsilon_j^*$, $j=1,2$, are bounded operators from $\cH$ to $\sH_{-2}$.

Next we introduce an abstract analog of the single and double layer potential (cf.\  \cite{Mc00}). For this it will be assumed that there is an abstract fundamental solution operator for $\cA$.

\begin{hypothesis}\label{h2}
There exists a bounded operator $\cG:\sH_{-2}\to\sH$ such that
\begin{equation}\label{eq:FundSol}
    \cG\cA f=f, \quad f\in\dom(T).
\end{equation}
\end{hypothesis}

\begin{definition}
Assume Hypotheses~\ref{h1} and \ref{h2} and let $\Upsilon_0^*,\Upsilon_1^*:\cH\to \sH_{-2}$ be
defined by \eqref{eq:Gamma_pm}. The abstract {\em single}
 and {\em double layer potentials} are defined by
 \begin{equation}\label{slpot}
 \cS \colon \begin{cases} \cH\rightarrow\sH, \\ 
 \varphi\mapsto \cG\Upsilon_0^* \varphi, \end{cases}
\end{equation}
and
\begin{equation}\label{dlpot}
 \cD \colon \begin{cases} \cH\rightarrow\sH, \\  
 \varphi\mapsto \cG\Upsilon_1^* \varphi, \end{cases} 
\end{equation}
respectively.
\end{definition}

It is clear that the operators $\cS$ and $\cD$
are well-defined and bounded. In order to obtain an abstract third Green identity in the next theorem
we will also use the following notations for the ``jumps'' of boundary values:
\begin{equation}\label{eq:Bracket1}
    [\Gamma_0f]:=\Gamma_0^+f_+-\Gamma_0^-f_-,\quad
    f=(f_+,f_-)^\top \in  \dom (T),
 \end{equation}
 and
 \begin{equation}\label{eq:Bracket2}
    [\Gamma_1f]:=\Gamma_1^+f_++\Gamma_1^-f_-, \quad
    f=(f_+,f_-)^\top \in  \dom (T).
 \end{equation}
One observes that the jump notations in \eqref{eq:Bracket1}--\eqref{eq:Bracket2} are compatible with the boundary conditions for elements in $\dom(A)$;
we note that different signs are used in \eqref{eq:Bracket2}
since in the application in Section~\ref{s5} the operators
$\Gamma^\pm_1$ will be the normal derivatives with the normals having opposite direction.

\begin{theorem} \lb{t3.7}
Assume Hypotheses~\ref{h1} and \ref{h2} and suppose that
$A$ in \eqref{aaaa} corresponds to the coupling of the quasi-boundary triples
$\{\cH,\Gamma_0^+,\Gamma_1^+\}$ and $\{\cH,\Gamma_0^-,\Gamma_1^-\}$ for $S_-^*$ and $S_+^*$, respectively. Let $[\Gamma_0\cdot]$, $[\Gamma_1\cdot]$, $\cS$ and $\cD$ be defined as above.
Then
\begin{equation}\label{eq:u}
f=\cG Tf+\cD[\Gamma_0 f] -\cS[\Gamma_1 f], \quad f \in \dom(T).
\end{equation}
\end{theorem}
\begin{proof}
 Let $f=(f_+,f_-)^\top\in\dom (T)$ and $g=(g_+,g_-)^\top\in\dom (A)$. Then it follows from \eqref{eq:cA} and the abstract Green identity \eqref{green1} that
 \begin{equation}\label{jussi2}
 \begin{split}
 {}_{\sH_{-2}}\langle\cA f,g\rangle_{\sH_2}&=(f,Ag)_\sH\\
 &=(f_+,T_+g_+)_{\sH_+}+(f_-,T_-g_-)_{\sH_-} \\
 &=(T_+f_+,g_+)_{\sH_+}-(\Gamma_1^+f_+,\Gamma_0^+g_+)_{\cH}+(\Gamma_0^+f_+,\Gamma_1^+g_+)_{\cH}\\
 &\quad + (T_-f_-,g_-)_{\sH_-}-(\Gamma_1^-f_-,\Gamma_0^-g_-)_{\cH}+(\Gamma_0^-f_-,\Gamma_1^-g_-)_{\cH}.
 \end{split}
 \end{equation}
As $g\in\dom(A)$ one concludes that
\begin{equation}
 \Upsilon_0 g=\Gamma_0^+g_+=\Gamma_0^-g_-, \quad
 \Upsilon_1 g=\Gamma_1^+g_+=-\Gamma_1^-g_-,
\end{equation}
by \eqref{gammagamma}, and hence \eqref{jussi2} takes on the form
 \begin{equation}\label{jussi3}
\begin{split}
{}_{\sH_{-2}}\langle\cA f,g\rangle_{\sH_2}&=(Tf,g)_{\sH}-(\Gamma_1^+f_+
+\Gamma_1^-f_-,\Upsilon_0 g)_{\cH}+(\Gamma_0^+f_+-\Gamma_0^-f_-,\Upsilon_1g)_{\cH}\\
 &=(Tf,g)_{\sH}-([\Gamma_1 f],\Upsilon_0 g)_{\cH}+([\Gamma_0 f],\Upsilon_1g)_{\cH} \\
 &= {}_{\sH_{-2}}\langle Tf,g\rangle_{\sH_2}
 - {}_{\sH_{-2}}\langle \Upsilon_0^* [\Gamma_1 f],g\rangle_{\sH_2}
 + {}_{\sH_{-2}}\langle\Upsilon_1^* [\Gamma_0 f], g\rangle_{\sH_2},
 \end{split}
 \end{equation}
 where \eqref{eq:Bracket1}--\eqref{eq:Bracket2} were used in the second equality, and \eqref{eq:Gamma_pm} was employed in the last equality.
 Since \eqref{jussi3} is true for all $g\in\dom(A)=\sH_2$ one concludes that
 \begin{equation}
  \cA f=Tf - \Upsilon_0^* [\Gamma_1 f] + \Upsilon_1^* [\Gamma_0 f],
 \end{equation}
and making use of the definition of $\cG$ and \eqref{slpot}--\eqref{dlpot} one finally obtains
\begin{equation}
 f=\cG\cA f=\cG Tf - \cG \Upsilon_0^* [\Gamma_1 f] + \cG \Upsilon_1^* [\Gamma_0 f]= \cG Tf-\cS[\Gamma_1 f] + \cD [\Gamma_0 f].
\end{equation}
\end{proof}

The following corollary can be viewed as an abstract unique continuation result.

\begin{corollary}\label{cor:Uniq}
Under the assumptions of Theorem~\ref{t3.7}, the following assertions
$(i)$,\,$(ii)$ hold: \\[1mm]
$(i)$ If $T_+ f_+=0$ for some $f_+\in\dom(T_+)$ and $\Gamma_0^+f_+=\Gamma_1^+f_+=0$ then $f_+=0$. \\[1mm]
$(ii)$ If $T_- f_-=0$ for some $f_-\in\dom(T_-)$ and $\Gamma_0^-f_-=\Gamma_1^-f_-=0$ then $f_-=0$.
\end{corollary}
\begin{proof}
We prove item $(i)$, the proof of assertion $(ii)$ being analogous. Assume that $T_+ f_+=0$ for some $f_+\in\dom(T_+)$ and $\Gamma_0^+f_+=\Gamma_1^+f_+=0$.
Setting $f_-=0$ one obtains for $f=(f_+ ,f_-)^\top\in\dom(T) $,
\begin{equation}
Tf =0, \quad
[\Gamma_0f]=\Gamma_0^+f_+-\Gamma_0^-f_-=0,\quad
[\Gamma_1f]=\Gamma_1^+f_++\Gamma_1^-f_-=0.
\end{equation}
Hence the third Green identity \eqref{eq:u} implies $f=0$ and therefore $f_+=0$.
\end{proof}

\section{Coupling of Schr\"{o}dinger Operators on Lipschitz Domains on Manifolds}  \lb{s5}

In this section we illustrate the abstract material in Sections \ref{s2}--\ref{s4} with
the concrete case of Schr\"odinger operators on Lipschitz domains on boundaryless Riemannian manifolds,
freely borrowing results from \cite{BGMM16}. For more details and background information concerning
differential geometry and partial differential equations on manifolds the interested reader is referred to
\cite{MMMT16}, \cite{MMT01}, \cite{Ta96}, and the literature cited there.

Suppose $(M,g)$ is a compact, connected, $C^\infty$, boundaryless Riemannian manifold of (real)
dimension $n\in\mathbb{N}$. In local coordinates, the metric tensor $g$ is expressed by
\begin{equation}\label{eqn.aaou}
g=\sum_{j,k=1}^n g_{jk}\,dx_j\otimes dx_k.
\end{equation}
As is customary, we shall use the symbol $g$ to also abbreviate
\begin{equation}\label{eqn.aaoik}
g:=\det\big[(g_{jk})_{1\leq j,k\leq n}\big],
\end{equation}
and we shall use $(g^{jk})_{1\leq j,k\leq n}$ to denote the inverse of
the matrix $(g_{jk})_{1\leq j,k\leq n}$, that is,
\begin{equation}\label{eqn.altou}
(g^{jk})_{1\leq j,k\leq n}:=\big[(g_{jk})_{1\leq j,k\leq n}\big]^{-1}.
\end{equation}
The volume element $d{\mathcal V}_{\!g}$ on $M$ with respect to the Riemannian
metric $g$ in \eqref{eqn.aaou} then can be written in local coordinates as
\begin{equation}\label{eqn.aa-pp}
d{\mathcal V}_{\!g}(x)=\sqrt{g(x)}\,d^n x.
\end{equation}

Following a common practice, we use $\{\partial_j\}_{1\leq j\leq n}$ to denote a local basis
in the tangent bundle $TM$ of the manifold $M$. This implies that if $X,Y\in TM$ are locally
expressed as $X=\sum_{j=1}^n X_j\partial_j$ and $Y=\sum_{j=1}^n Y_j\partial_j$, then
\begin{equation}\label{eqn.aa-phREED}
\langle X,Y\rangle_{TM}=\sum_{j,k=1}^n X_jY_k g_{jk},
\end{equation}
where $\langle\cdot,\cdot\rangle_{TM}$ stands for the pointwise inner product in $TM$.

Next, we discuss the gradient and divergence operators associated with the metric $g$ on the manifold $M$.
Specifically, given an open set $\Omega\subset M$ and some function $f\in C^1(\Omega)$,
the gradient of $f$ is the vector field locally defined as
\begin{equation}\label{rdf96-2DC.1RR}
{\rm grad}_g (f):=\sum_{j,k=1}^n (\partial_jf)g^{jk}\partial_k.
\end{equation}
Also, given any vector field $X\in C^1(\Omega,TM)$ locally written as $X=\sum_{j=1}^n X_j\partial_j$,
its divergence is given by
\begin{eqnarray}\label{gMM-24GBN}
{\rm div}_g (X):=\sum_{j=1}^n g^{-1/2}\partial_j(g^{1/2}X_j)
=\sum_{j=1}^n\partial_jX_j+\sum_{j,k=1}^n\Gamma^j_{jk}X_k,
\end{eqnarray}
where $\Gamma^i_{jk}$ are the Christoffel symbols associated with the metric \eqref{eqn.aaou}. The Laplace--Beltrami operator
\begin{equation}\label{eqn.thu-TELL}
\Delta_g:={\rm div}_g\,{\rm grad}_g,
\end{equation}
is expressed locally as
\begin{equation}\label{eqn.thuaou}
\Delta_gu=\sum_{j,k=1}^n g^{-1/2}\partial_j\big(g^{jk}g^{1/2}\partial_ku\big).
\end{equation}
We are interested in working with the Schr\"odinger operator
\begin{equation}\label{eqn.lmnae}
L:=-\Delta_g+V,
\end{equation}
where the potential $V\in L^\infty(M)$ is a real scalar-valued function.

The reader is reminded that the scale of $L^2$-based Sobolev spaces $H^s(M)$ of fractional smoothness $s\in{\mathbb{R}}$ on $M$ may be defined in a natural fashion, via localization (using a smooth partition of unity subordinate to a finite cover of $M$ with local coordinate charts) and pull-pack to the Euclidean model. This scale of spaces is then adapted to an open subset $\Omega$ of $M$ via restriction, by setting
\begin{equation}\label{u64rLL-1}
H^s(\Omega):=\big\{u\big|_{\Omega} \,\big|\, u\in H^s(M)\big\},\quad s\in{\mathbb{R}}.
\end{equation}
In particular, $H^0(\Omega)$ coincides with $L^2(\Omega)$, the space of square-integrable functions
with respect to the volume element $d{\mathcal V}_{\!g}$ in $\Omega$. For each
$s\in{\mathbb{R}}$ we also define
\begin{equation}\label{tr6ytG-MMM}
\mathring{H}^s(\Omega):=\overline{C^\infty_0(\Omega)}^{H^s(\Omega)},
\end{equation}
and equip the latter space with the norm inherited from $H^s(\Omega)$.

Since bounded Lipschitz domains in the Euclidean setting are invariant under $C^1$
diffeomorphisms (cf.\ \cite{HMT07}), this class may canonically be defined on the
manifold $M$, using local coordinate charts. Given a Lipschitz domain $\Omega$, it is then possible to define (again, in a canonical manner, via localization and pull-back) fractional Sobolev spaces on its boundary, $H^s(\partial\Omega)$, for $s\in[-1,1]$. In such a scenario one has
\begin{equation}
\big(H^s(\partial\Omega)\big)^\ast=H^{-s}(\partial\Omega), \quad s\in[-1,1],
\end{equation}
and $H^0(\partial\Omega)$ coincides with $L^2(\partial\Omega)$, the space of square-integrable functions with respect to the surface measure $\sigma_{\!g}$ induced by the ambient Riemannian metric on $\partial\Omega$. Moreover,
\begin{equation}\label{sop-uGG}
\big\{\phi\big|_{\partial\Omega} \, \big| \,\phi\in C^\infty(M)\big\}\,\text{ is dense in each }\,
H^s(\partial\Omega),\quad s\in[-1,1],
\end{equation}
and
\begin{equation}\label{sop-uGG.2}
H^{s_0}(\partial\Omega)\hookrightarrow H^{s_1}(\partial\Omega)\,\text{ continuously, whenever }\,
-1\leq s_1\leq s_0\leq 1.
\end{equation}

In the following the operator $A$ in Sections \ref{s3} and \ref{s4} will be
the Schr\"{o}dinger operator
\begin{equation}\label{sop}
A:=-\Delta_g +V, \quad \dom(A):=H^2(M).
\end{equation}
To proceed, we fix a Lipschitz domain $\Omega_{+}\subset M$ and denote
\begin{equation}\label{i6RR}
\Omega_{-}:=M\backslash\overline{\Omega_{+}}.
\end{equation}
Then $\Omega_{-}$ is also a Lipschitz domain, sharing a common compact boundary with $\Omega_{+}$,
\begin{equation}\label{u64dddf}
\cC:=\partial\Omega_{+}=\partial\Omega_{-}.
\end{equation}
At the global level, it is important to note that $A$ is selfadjoint in the Hilbert
space $\sH=L^2(M)$. The decomposition of $M\backslash\cC$ into the disjoint
union of $\Omega_{+}$ and $\Omega_{-}$, induces a direct orthogonal sum  decomposition of $\sH$ into two Hilbert spaces $\sH_+$ and $\sH_-$, defined as $\sH_\pm:=L^2(\Omega_\pm)$,
\begin{equation}
\sH = L^2(\Omega_+) \oplus L^2(\Omega_-).
\end{equation}
In the following functions on $M$ will be identified with
the pair of restrictions onto $\Omega_+$ and $\Omega_-$ and a vector notation will be used. For example, for $f\in L^2(M)$ we shall also write $(f_+,f_-)^\top$, where
$f_\pm\in L^2(\Omega_\pm)$. This notation is in accordance with the
notation in Sections~\ref{s3} and \ref{s4}. Also, in the sequel we agree to abbreviate
\begin{equation}\label{u64dddf-y5r}
V_{\pm}:=V\big|_{\Omega_\pm}\in L^\infty(\Omega_\pm).
\end{equation}

For $s\geq 0$ we define the Banach spaces
\begin{equation}\label{y555}
H^s_\Delta(\Omega_\pm):=\big\{f_\pm\in H^s(\Omega_\pm) \, \big| \,\Delta_g f_\pm\in L^2(\Omega_\pm)\big\},
\end{equation}
equipped with the norms $\|\cdot\|_{H^s_\Delta(\Omega_\pm)}$ defined as
\begin{equation}\label{65rtf}
\|f_\pm\|_{H^s_\Delta(\Omega_\pm)}:=\|f_\pm\|_{H^s(\Omega_\pm)}
+\big\|(-\Delta_g+V_\pm)f_\pm\big\|_{L^2(\Omega_\pm)}, \quad
f_\pm\in H^s_\Delta(\Omega_\pm).
\end{equation}
The minimal and maximal realizations of $-\Delta_g+V_\pm$ in $L^2(\Omega_\pm)$ are defined by
\begin{equation}\label{y55rr-34}
S_{min,\pm}:=-\Delta_g+V_{\pm},\quad\dom(S_{min,\pm})=\mathring{H}^2(\Omega_\pm),
\end{equation}
and
\begin{equation}\label{75rrff}
S_{max,\pm}:=-\Delta_g+V_{\pm},\quad\dom(S_{max,\pm})=H^0_\Delta(\Omega_\pm).
\end{equation}

In the next lemma we collect some well-known properties of the operators
$S_{min,\pm}$ and $S_{max,\pm}$. A proof of this lemma and some further
properties of the minimal and maximal realization of $-\Delta_g+V_\pm$ can
be found, for instance, in \cite{BGMM16}.

\begin{lemma}\label{LLahvfv}
The operators $S_{min,\pm}$ and $S_{max,\pm}$ are densely defined and closed in
$L^2(\Omega_\pm)$. The operator $S_{min,\pm}$ is symmetric, semibounded from below,
and has infinite deficiency indices. Furthermore, $S_{min,\pm}$ and $S_{max,\pm}$
are adjoints of each other, that is,
\begin{equation}\label{7544rr}
\big(S_{min,\pm}\big)^*=S_{max,\pm}\,\text{ and } \,S_{min,\pm}=\big(S_{max,\pm}\big)^*.
\end{equation}
\end{lemma}

Let $\mathfrak{n}^\pm\in L^\infty(\cC,TM)$ be the outward unit normal vectors to $\Omega_\pm$.
One observes that in the present situation $\mathfrak{n}^+=-\mathfrak{n}^-$.
The Dirichlet and Neumann trace operators $\tau^\pm_D$ and $\tau^\pm_N$, originally defined by
\begin{equation}\label{uy55rff-54}
\tau^\pm_D f_\pm:=f_\pm\upharpoonright_{\cC},\quad
\tau^\pm_N f_\pm:=\big\langle\mathfrak{n}^\pm,{\rm grad}_g f_\pm\upharpoonright_{\cC}\big\rangle_{TM}
\end{equation}
for $f_\pm\in C^\infty(\overline{\Omega_\pm})$, admit continuous linear extensions to operators
\begin{equation}\label{trace}
\tau^\pm_D :H^s_\Delta(\Omega_\pm)\to H^{s-1/2}(\cC) \, \text{ and } \,
\tau^\pm_N : H^s_\Delta(\Omega_\pm)\to H^{s-3/2}(\cC), 
\end{equation}
whose actions are compatible with one another, for all $s\in[\tfrac{1}{2},\tfrac{3}{2}]$.
We refer to \cite{BGMM16} where it is also shown that
\begin{align}\label{tra-trf}
& \text{the trace operators $\tau^\pm_D$,\, $\tau^\pm_N$ in \eqref{trace}} \\
& \quad \text{ are both surjective for each $s\in[\tfrac{1}{2},\tfrac{3}{2}]$.}
\end{align}
We wish to augment \eqref{tra-trf} with the following density result.

\begin{lemma}\label{ytFVVa-LL}
The ranges of the mappings
\begin{align}\label{Utrafr-1}
\big\{f_{\pm}\in H^{3/2}(\Omega_{\pm}) \, \big| \, 
\Delta_g f_\pm\in C^\infty(\,\overline{\Omega_{\pm}}\,)\big\}
\ni f_\pm\mapsto\big(\tau^\pm_D f_\pm,-\tau^\pm_N f_\pm\big)
\end{align}
are dense in $L^2(\cC)\times L^2(\cC)$.
\end{lemma}

The proof of Lemma~\ref{ytFVVa-LL} requires some preparations. To get started, we fix
two potentials $0\leq V^\pm_0\in C^\infty(M)$ which are not identically zero on $M$,
and which vanish on $\Omega_{+}$, and on $\Omega_{-}$, respectively.
Then (cf. \cite[p.~27]{MT00b}) for each $s\in[0,2]$, the operators
\begin{align}\label{khaba-GGG-XXX}
& -\Delta_g+V^\pm_0:H^{2-s}(M)\longrightarrow H^{-s}(M) \\
& \quad \, \text{ are invertible, with bounded inverses.}
\end{align}
Furthermore, for each choice of sign, the said inverses act in a compatible fashion with one another.
Abbreviating $\cG^\pm_0:=(-\Delta_g+V^\pm_0)^{-1}$ then yields two well-defined, linear,
and bounded operators
\begin{equation}\label{khaba-GGG-XXX.2}
\cG^\pm_0:H^{-s}(M)\longrightarrow H^{2-s}(M) \,\text{ for each } s\in[0,2].
\end{equation}
The Schwartz kernels of these operators are distributions $E^\pm_0$ on $M\times M$ which are smooth
outside of the diagonal ${\rm diag}M:=\{(x,x):\,x\in M\}$. In particular, it makes sense
to talk about pointwise values $E^\pm_0(x,y)$ for $x,y\in M$ with $x\not= y$. Among other things,
the functions $E^\pm_0(\cdot,\cdot)\in C^\infty\big((M\times M)\backslash{\rm diag}M\big)$ satisfy
\begin{equation}\label{ju6gfc}
E^\pm_0(x,y)=E^\pm_0(y,x)\,\,\text{ for all }\,\,x,y\in M\,\,\text{ with }\,\,x\not=y.
\end{equation}

At this stage, we bring in the single and double layer potentials on Lipschitz domains on manifolds
considered in \cite{MT99}--\cite{MT00b}. Their actions on an arbitrary function $\varphi\in L^2(\cC)$
are, respectively,
\begin{equation}\label{slpot2-XXX-YYa}
(\cS^{\pm}_0\varphi)(x):=\int_{\cC}E^\pm_0(x,y)\varphi(y)\,d\sigma_{\!g}(y),\quad x\in\Omega_{\pm},
\end{equation}
and
\begin{equation}\label{dlpot2-XXX-YYa}
(\cD^{\pm}_0\varphi)(x):=\pm\int_{\cC}\big\langle\mathfrak{n}^{\pm}(y),{\rm grad}_{g_y}[E^\pm_0(x,y)]\big\rangle_{T_yM}
\varphi(y)\,d\sigma_{\!g}(y),\quad x\in\Omega_\pm,
\end{equation}
where $\sigma_{\!g}$ is the surface measure induced by the ambient Riemannian metric on $\cC$.
Let us also consider their boundary versions, that is, the singular integral operators acting
on an arbitrary function $\varphi\in L^2(\cC)$ according to
\begin{equation}\label{slpot2-XXX-YYa.2}
(S^\pm_0\varphi)(x):=\int_{\cC}E^\pm_0(x,y)\varphi(y)\,d\sigma_{\!g}(y),\quad x\in\cC,
\end{equation}
and
\begin{equation}\label{dlpot2-XXX-YYa.2}
(K^\pm_0\varphi)(x):={\rm P.V.}\int_{\cC}
\big\langle\mathfrak{n}^{\pm}(y),{\rm grad}_{g_y}[E^\pm_0(x,y)]\big\rangle_{T_yM}
\varphi(y)\,d\sigma_{\!g}(y),\quad x\in\cC,
\end{equation}
where ${\rm P.V.}$ indicates that the integral is considered in the principal value sense
(i.e., removing a small geodesic ball centered at the singularity and passing to the limit
as its radius shrinks to zero). Work in \cite{MT99}--\cite{MT00b} ensures that the following
properties hold:
\begin{align}\label{PRop-AA.6}
& \cD^{\pm}_0:L^2(\cC)\to H^{1/2}_\Delta(\Omega_{\pm})\,\,\text{ are linear and bounded operators},
\\[4pt]
& (-\Delta_g+V^\pm_0)(\cD^{\pm}_0\varphi)=0\,\,\text{ in $\Omega_\pm$, for every function }\,\,\varphi\in L^2(\cC),
\label{PRop-AA.7}
\\[4pt]
& K^\pm_0:L^2(\cC)\to L^2(\cC)\,\,\text{ are linear and bounded operators},
\label{PRop-AA.8}
\\[4pt]
& \tau_D^\pm(\cD^{\pm}_0\varphi)=\pm\big(\tfrac{1}{2}I+K_0^\pm\big)\varphi\,\,
\text{ on $\cC$, for each }\,\,\varphi\in L^2(\cC),
\label{PRop-AA.9}
\\[4pt]
& \cS^{\pm}_0:L^2(\cC)\to H^{3/2}_\Delta(\Omega_{\pm})\,\,\text{ are linear and bounded operators},
\label{PRop-AA.1}
\\[4pt]
& (-\Delta_g+V^\pm_0)(\cS^{\pm}_0\varphi)=0\,\,\text{ in $\Omega_\pm$, for every function }\,\,\varphi\in L^2(\cC),
\label{PRop-AA.2}
\\[4pt]
& S^\pm_0:L^2(\cC)\to L^2(\cC)\,\,\text{ are linear, bounded, self-adjoint, and injective},
\label{PRop-AA.3}
\\[4pt]
& \tau_D^\pm(\cS^{\pm}_0\varphi)=S^\pm_0\varphi\,\,\text{ on $\cC$, for each function }\,\,\varphi\in L^2(\cC),
\label{PRop-AA.4}
\\[4pt]
& \tau_N^\pm(\cS^{\pm}_0\varphi)=\big(-\tfrac{1}{2}I+(K_0^\pm)^\ast\big)\varphi\,\,
\text{ on $\cC$, for every }\,\,\varphi\in L^2(\cC),
\label{PRop-AA.5}
\end{align}
where $I$ is the identity operator on $L^2(\cC)$, and $(K_0^\pm)^\ast$ are the adjoints of the
operators $K_0^\pm$ in \eqref{PRop-AA.8}. After this preamble, we are ready to
present the proof of Lemma~\ref{ytFVVa-LL}.

\vskip 0.08in
\begin{proof}[Proof of Lemma~\ref{ytFVVa-LL}]
The density results claimed in the statement follow
as soon as we establish that whenever two functions $h_D,h_N\in L^2(\cC)$ satisfy
\begin{align}\label{Utrafr-GF}
& \big(\tau^\pm_D f_\pm,h_D\big)_{L^2(\cC)}-\big(\tau^\pm_N f_\pm,h_N\big)_{L^2(\cC)}=0 \\
& \quad \text{ for all }\,f_\pm\in H^{3/2}(\Omega_{\pm})\,\text{ with }
\,\Delta_g f_\pm\in C^\infty(\,\overline{\Omega_{\pm}}\,),
\end{align}
then necessarily $h_D=0$ and $h_N=0$. To this end, pick an arbitrary
$h\in L^2(\cC)$ and consider $f_\pm:=\cS^\pm_0 h$ in $\Omega_{\pm}$.
Then $f_\pm\in H^{3/2}_\Delta(\Omega_{\pm})$ due to \eqref{PRop-AA.1}.
Also, relying on \eqref{PRop-AA.2} and the fact that, by design, the potentials
$V^\pm_0$ vanish in $\Omega_\pm$, we may write
$\Delta_g f_\pm=(-\Delta_g+V^\pm_0)f_\pm=0$ in $\Omega_\pm$.
Granted these properties of $f_\pm$, from \eqref{Utrafr-GF}, \eqref{PRop-AA.4}, \eqref{PRop-AA.5},
and \eqref{PRop-AA.3}, one concludes that
\begin{align}\label{Utrafr-GF-Tga.1}
0 &=\big(\tau^\pm_D f_\pm,h_D\big)_{L^2(\cC)}-\big(\tau^\pm_N f_\pm,h_N\big)_{L^2(\cC)}
\nonumber\\
&=\big(S^\pm_0 h,h_D\big)_{L^2(\cC)}-\Big(\big(-\tfrac{1}{2}I+(K_0^\pm)^\ast\big)h\,,\,h_N\Big)_{L^2(\cC)}
\nonumber\\
&=\big(h,S^\pm_0 h_D\big)_{L^2(\cC)}-\big(h\,,\,(-\tfrac{1}{2}I+K_0^\pm)h_N\big)_{L^2(\cC)}
\nonumber\\
&=\big(h\,,\,S^\pm_0 h_D-(-\tfrac{1}{2}I+K_0^\pm)h_N\big)_{L^2(\cC)}.
\end{align}
With this in hand, the arbitrariness of $h\in L^2(\cC)$ then forces
\begin{align}\label{Utrafr-GF-Tga.2}
S^\pm_0 h_D-(-\tfrac{1}{2}I+K_0^\pm)h_N=0.
\end{align}

Next, we pick two arbitrary functions $\phi_\pm\in C^\infty_0(\Omega_\pm)$ and, this time,
consider $f_\pm:=\cG_0^\pm\phi_\pm$ in $\Omega_\pm$. Then \eqref{khaba-GGG-XXX.2} 
ensures that $f_\pm\in H^2(\Omega_\pm)$. Given that, by design, $V^\pm_0$ vanish in 
$\Omega_\pm$, one also has 
$\Delta_g f_\pm=(-\Delta_g+V^\pm_0)\cG_0^\pm\phi_\pm=\phi_\pm$ in $\Omega_\pm$.
Having established these properties of $f_\pm$, \eqref{Utrafr-GF} implies that
\begin{align}\label{Utrafr-GF-Tga.7}
\big(\tau^\pm_D f_\pm,h_D\big)_{L^2(\cC)}-\big(\tau^\pm_N f_\pm,h_N\big)_{L^2(\cC)}=0.
\end{align}
Now we take a closer look at the two terms in the left-hand side of \eqref{Utrafr-GF-Tga.7}.
For the first term we write
\begin{align}\label{Utrafr-GF-Tga.8}
\big(\tau^\pm_D f_\pm,h_D\big)_{L^2(\cC)}
&=\int_{\cC}\Big(\int_{\Omega_\pm}E^{\pm}_0(x,y)\overline{\phi_\pm(y)}\,d{\mathcal V}_{\!g}(y)\Big)
h_D(x)\,d\sigma_{\!g}(x)
\nonumber\\
&=\int_{\Omega_\pm}\Big(\int_{\cC}E^{\pm}_0(x,y)h_D(x)\,d\sigma_{\!g}(x)\Big)
\overline{\phi_\pm(y)}\,d{\mathcal V}_{\!g}(y)
\nonumber\\
&=\int_{\Omega_\pm}(\cS^\pm_0h_D)(y)\overline{\phi_\pm(y)}\,d{\mathcal V}_{\!g}(y),
\end{align}
where the first equality uses the definition of $E^{\pm}_0(x,y)$, the second equality
is based on Fubini's theorem, while the third equality is a consequence of \eqref{ju6gfc}
and \eqref{slpot2-XXX-YYa}. For the second term in \eqref{Utrafr-GF-Tga.7} we compute
\begin{align}\label{Utrafr-GF-Tga.9}
& \big(\tau^\pm_N f_\pm,h_N\big)_{L^2(\cC)}
\nonumber\\
& \quad 
=\int_{\cC}\Big(\int_{\Omega_\pm}\big\langle\mathfrak{n}^{\pm}(x),{\rm grad}_{g_x}[E^\pm_0(x,y)]\big\rangle_{T_xM}
\overline{\phi_\pm(y)}\,d{\mathcal V}_{\!g}(y)\Big)
h_N(x)\,d\sigma_{\!g}(x)
\nonumber\\
& \quad 
=\int_{\Omega_\pm}\Big(\int_{\cC}\big\langle\mathfrak{n}^{\pm}(x),{\rm grad}_{g_x}[E^\pm_0(x,y)]\big\rangle_{T_xM}
h_N(x)\,d\sigma_{\!g}(x)\Big)
\overline{\phi_\pm(y)}\,d{\mathcal V}_{\!g}(y)
\nonumber\\
& \quad
=\pm\int_{\Omega_\pm}(\cD^\pm_0h_N)(y)\overline{\phi_\pm(y)}\,d{\mathcal V}_{\!g}(y),
\end{align}
where the first equality relies on the definition of $E^{\pm}_0(x,y)$, the second equality
uses Fubini's theorem, while the third equality is implied by \eqref{ju6gfc}
and \eqref{dlpot2-XXX-YYa}.

Together, \eqref{Utrafr-GF-Tga.7}, \eqref{Utrafr-GF-Tga.8}, and \eqref{Utrafr-GF-Tga.9} imply that
\begin{align}\label{Utrafr-GF-Tga.10}
\int_{\Omega_\pm}\Big\{(\cS^\pm_0h_D)(y)\mp(\cD^\pm_0h_N)(y)\Big\}\overline{\phi_\pm(y)}\,d{\mathcal V}_{\!g}(y)=0
\end{align}
which, in view of the arbitrariness of $\phi_\pm\in C^\infty_0(\Omega_\pm)$, forces
\begin{align}\label{Utrafr-GF-Tga.11}
\cS^\pm_0h_D\mp\cD^\pm_0h_N=0 \,\text{ in } \,\Omega_\pm.
\end{align}
Applying $\tau_D^\pm$ to both sides of \eqref{Utrafr-GF-Tga.11} then yields,
on account of \eqref{PRop-AA.4} and \eqref{PRop-AA.9},
\begin{align}\label{Utrafr-GF-Tga.12}
S^\pm_0h_D-\big(\tfrac{1}{2}I+K_0^\pm\big)h_N=0 \,\text{ on } \,\cC.
\end{align}

The end-game in the proof of the lemma is as follows.
Subtracting \eqref{Utrafr-GF-Tga.12} from \eqref{Utrafr-GF-Tga.2} proves that $h_N=0$.
Using this back into \eqref{Utrafr-GF-Tga.12} leads to $S^\pm_0h_D=0$ which, in light
of the injectivity of the single layer operators in \eqref{PRop-AA.3}, shows that $h_D=0$
as well. Hence, $h_D=h_N=0$, as desired.
\end{proof}

Going further, in the next theorem we define quasi boundary triples for $S_{max,\pm}=\big(S_{min,\pm}\big)^*$
with the natural trace maps as boundary maps defined on the domain of the operators
\begin{equation}\label{tpm}
T_\pm:=-\Delta_g+V_\pm,\quad\dom(T_\pm):=H^{3/2}_\Delta(\Omega_\pm),
\end{equation}
in $L^2(\Omega_{\pm})$. One recalls that
\begin{equation}\label{tpma}
T=\begin{pmatrix} T_+ & 0 \\ 0 & T_-\end{pmatrix}  \, \text{ in } \,
L^2(M) = L^2(\Omega_+) \oplus L^2(\Omega_-).
\end{equation}

With this choice of boundary maps the values of the corresponding Weyl function are
Dirichlet-to-Neumann maps (up to a minus sign).

\begin{theorem}\label{qbtlip}
Let $\Omega_\pm$ and $T_\pm$ be as above, and let
\begin{equation}\label{Yrafac}
\Gamma_0^\pm,\Gamma_1^\pm:H^{3/2}_\Delta(\Omega_\pm)\to L^2(\cC),\quad
\Gamma_0^\pm f_\pm:=\tau^\pm_D f_\pm,\quad\Gamma^\pm_1 f_\pm:=-\tau^\pm_N f_\pm.
\end{equation} 	
Then $\{L^2(\cC),\Gamma^\pm_0,\Gamma^\pm_1\}$ are quasi boundary triples for
$T_\pm\subset S_{max,\pm}$ such that
\begin{equation}\label{u655tgf}
S_{min,\pm}=T_\pm\upharpoonright\big(\ker(\Gamma_0^\pm)\cap\ker(\Gamma_1^\pm)\big).
\end{equation}
In addition, the following statements $(i)$--$(iii)$ hold: \\[1mm]
$(i)$ The Dirichlet realizations $A_{D,\pm}$ and the Neumann realizations
$A_{N,\pm}$ of $-\Delta_g+V_\pm$ in $L^2(\Omega_\pm)$ coincide with
$A_{0,\pm}=T\upharpoonright\ker(\Gamma^\pm_0)$ and
$A_{1,\pm}=T\upharpoonright\ker(\Gamma^\pm_1)$, respectively,
\begin{align}\label{u655-4r}
A_{D,\pm} &=T_\pm\upharpoonright\ker(\Gamma^\pm_0)=(-\Delta_g+V_\pm)\upharpoonright
\big\{f_\pm\in H^{3/2}_\Delta(\Omega_\pm)\,\big|\,\tau^\pm_D f_\pm=0\big\},
\\[1mm]
A_{N,\pm} &=T_\pm\upharpoonright\ker(\Gamma^\pm_1)=(-\Delta_g+V_\pm)\upharpoonright
\big\{f_\pm\in H^{3/2}_\Delta(\Omega_\pm)\,\big|\,\tau^\pm_N f_\pm=0\big\},
\end{align}
and both operators are self-adjoint in $L^2(\Omega_\pm)$. \\[1mm]
$(ii)$ The values $\gamma_\pm(z):L^2(\cC)\supset H^1(\cC)\to L^2(\Omega_\pm)$
of the $\gamma$-fields are given by
\begin{equation}\label{gammas}
\gamma_\pm(z)\varphi=f_\pm,\quad\varphi\in H^1(\cC),
\quad z\in\rho(A_{D,\pm}),
\end{equation}
where $f_\pm\in L^2(\Omega_\pm)$ are the unique solutions of the boundary value problems
\begin{equation}\label{bvp}
f_\pm\in H^{3/2}_\Delta(\Omega_\pm),\quad
(-\Delta_g+V_\pm-z)f_\pm=0,\quad\tau^\pm_D f_\pm=\varphi.
\end{equation}
$(iii)$ The values $M_\pm (z):L^2(\cC)\supset H^1(\cC)\to L^2(\cC)$
of the Weyl functions  are Dirichlet-to-Neumann maps, given by
\begin{equation}\label{dnmaps}
M_\pm(z)\varphi=-\tau^\pm_N f_\pm,\quad\varphi\in H^1(\cC),
\quad z\in\rho(A_{D,\pm}),
\end{equation}
where $f_\pm=\gamma_\pm(z)\varphi$ are the unique solutions of \eqref{bvp}.
\end{theorem}
\begin{proof}
Let us verify the properties stipulated in Definition~\ref{Def-triple} in the current case.
First, the abstract Green identity \eqref{green1} presently corresponds to the second Green
identity for the Schr\"odinger operator \eqref{eqn.lmnae} on the Lipschitz domain $\Omega$,
proved in \cite{BGMM16}. Second, the fact that $\ran(\Gamma_0^\pm,\Gamma_1^\pm)^\top$ is dense
in $L^2(\cC)\times L^2(\cC)$ is readily implied by Lemma~\ref{ytFVVa-LL} (bearing in mind \eqref{y555}).
Third, the self-adjointness of $A_{0,\pm}=T_\pm\upharpoonright\ker(\Gamma_0^\pm)$ is clear from the
fact that these operators coincide with the self-adjoint Dirichlet realizations of $-\Delta_g+V_{\pm}$
in $\Omega_{\pm}$ studied in \cite{BGMM16}.

Fourth, we focus on establishing that $\overline{T}_\pm=S_{max,\pm}$. In turn, since
${\rm dom}(T_\pm)$ contains ${\rm dom}(A_{D,\pm})+{\rm dom}(A_{N,\pm})$, this is going to be a consequence of the fact that
\begin{align}\label{uyvvva8h}
\begin{split}
& \text{${\rm dom}(A_{D,\pm})+{\rm dom}(A_{N,\pm})$ is dense in ${\rm dom}(S_{max,\pm})$} \\
& \quad \text{ with respect to the graph norm.}
\end{split}
\end{align}
To prove \eqref{uyvvva8h}, we assume that
$h\in {\rm dom}(S_{max,\pm})$ is such that
\begin{align}\label{uyvvva8h.2}
\begin{split}
& (f_D+f_N,h)_{L^2(\Omega_{\pm})}+(S_{max,\pm}(f_D+f_N),S_{max,\pm}h)_{L^2(\Omega_{\pm})}=0
\\
& \quad \text{ for all }\,\,f_D\in{\rm dom}(A_{D,\pm}), \, f_N\in{\rm dom}(A_{N,\pm}).
\end{split}
\end{align}
Then
\begin{equation}\label{uyvvva8h.2A}
(A_{D,\pm}f_D,S_{max,\pm}h)_{L^2(\Omega_{\pm})}=(f_D,-h)_{L^2(\Omega_{\pm})}, 
\quad f_D\in{\rm dom}(A_{D,\pm}),
\end{equation}
and
\begin{equation}\label{uyvvva8h.2B}
(A_{N,\pm}f_N,S_{max,\pm}h)_{L^2(\Omega_{\pm})}=(f_N,-h)_{L^2(\Omega_{\pm})}, 
\quad f_N\in{\rm dom}(A_{N,\pm}).
\end{equation}
Together, \eqref{uyvvva8h.2A} and \eqref{uyvvva8h.2B} prove that
\begin{align}\label{uyvvva8h.3}
\begin{split}
& S_{max,\pm}h\in{\rm dom}(A_{D,\pm})\cap{\rm dom}(A_{N,\pm})={\rm dom}(S_{min,\pm}) \\
& \quad \text{ and }\,\,S_{min,\pm}S_{max,\pm}h=-h.
\end{split}
\end{align}
Finally, from
\begin{align}\label{uyvvva8h.4}
\begin{split}
0 &=\big((I+S_{min,\pm}S_{max,\pm})h,h\big)_{L^2(\Omega_{\pm})} \\
&=(h,h)_{L^2(\Omega_{\pm})}+(S_{max,\pm}h,S_{max,\pm}h)_{L^2(\Omega_{\pm})},
\end{split}
\end{align}
one concludes that $h=0$. Hence \eqref{uyvvva8h} holds, completing the proof of the fact that
$\overline{T}_\pm=S_{max,\pm}$.

This shows that $\{L^2(\cC),\Gamma^\pm_0,\Gamma^\pm_1\}$ are indeed
quasi boundary triples for $T_\pm$. That $T_\pm\subset S_{max,\pm}$ is clear from definitions,
while \eqref{u655tgf} has been established in \cite{BGMM16}.

Thanks to work in \cite{BGMM16}, the assertions in $(ii)$ and $(iii)$ follow immediately
from the definition of the $\gamma$-field and the Weyl function. We refer the interested
reader to \cite{BGMM16} for more details. Here we only wish to note that in the case of
a bounded Lipschitz domain in the flat Euclidean setting (i.e., ${\mathbb{R}}^n$ equipped
with the standard metric) a similar result has been established in \cite[Theorem 4.1]{BM14}.
\end{proof}

In the following we establish the link to the coupling procedure discussed in Section~\ref{s3}.
First of all we set $\cH:=L^2(\cC)$ so that the quasi boundary triples in Section~\ref{s3} are those
in Theorem~\ref{qbtlip}. The operator $S$ in \eqref{sts} is the direct orthogonal
sum of the minimal realizations $S_{min,+}$ and $S_{min,-}$,
\begin{equation}\label{i6tffF}
S=\begin{pmatrix} S_{min,+} & 0 \\ 0 & S_{min,-}\end{pmatrix},
\end{equation}
and the boundary mappings in the quasi boundary triple
$\{L^2(\cC)\oplus L^2(\cC),\Gamma_0,\Gamma_1\}$ in \eqref{qbtcoup} are now given by
\begin{equation}\label{u8bvcca}
\Gamma_0 f=\begin{pmatrix} \tau^+_D f_+ \\ \tau^-_D f_- \end{pmatrix}
\, \text{ and } \,
\Gamma_1 f=\begin{pmatrix} -\tau^+_N f_+ \\ -\tau^-_N f_- \end{pmatrix},
\end{equation}
where $f=(f_+,f_-)\in\dom(T)$ with $T_{\pm}$, $T$ given as in \eqref{tpm},
\eqref{tpma}. The self-adjoint operator corresponding to $\ker(\Gamma_0)$
is the orthogonal sum of the Dirichlet operators $A_{D,+}$ and $A_{D,-}$ in $L^2(\Omega_+)$
and $L^2(\Omega_-)$, respectively, 
\begin{equation}\label{add}
A_0=\begin{pmatrix} A_{D,+} & 0 \\ 0 & A_{D,-}\end{pmatrix}.
\end{equation}

The following lemma shows that the coupling of the quasi boundary triples
$\{L^2(\cC),\Gamma^+_0,\Gamma^+_1\}$ and $\{L^2(\cC),\Gamma^-_0,\Gamma^-_1\}$
in Theorem~\ref{resthm} leads to the self-adjoint Schr\"{o}\-din\-ger operator in \eqref{sop}.

\begin{lemma}\label{tapasst}
The operator
\begin{equation}\label{ta}
T\upharpoonright\big\{f=(f_+,f_-)^\top\in\dom(T) \, \big| \,\Gamma^+_0f_+=\Gamma_0^-f_-,\,
\Gamma^+_1f_+=-\Gamma_1^-f_-\big\}
\end{equation}
coincides with the self-adjoint operator $A$ in \eqref{sop}.
\end{lemma}
\begin{proof}
Since any function $f\in H^2(M)$ satisfies
\begin{align}
\Gamma^+_0f_+ &= \tau_D^+f_+=\tau_D^-f_-=\Gamma_0^-f_-,  \label{uyt5rff} \\
\Gamma^+_1f_+ &= -\tau_N^+f_+=\tau_N^-f_-=-\Gamma_1^-f_-,  \label{yr3edda}
\end{align}
and $f_\pm\in H^2(\Omega_\pm)\subset H^{3/2}_\Delta(\Omega_\pm)=\dom(T_\pm)$, 
it follows that $H^2(M)=\dom(A)$ is contained in the domain of the operator in
\eqref{ta}. On the other hand, it follows from Theorem~\ref{resthm}\,$(i)$ that the
operator in \eqref{ta} is symmetric, and hence self-adjoint (as it extends the self-adjoint operator $A$).
\end{proof}

As an immediate consequence of the observation in Lemma~\ref{tapasst} we obtain the
next corollary. First, we note that the self-adjointness of the operator $A$ in \eqref{sop},
Theorem~\ref{resthm}, and the fact that
\begin{equation}\label{ta2}
A=T\upharpoonright\big\{f=(f_+,f_-)^\top\in\dom(T)\, \big| \,\Gamma^+_0f_+=\Gamma_0^-f_-,\,
\Gamma^+_1f_+=-\Gamma_1^-f_- \big\}, 
\end{equation}
imply
\begin{equation}\label{uy53wwws}
\ran\big(\Gamma^\pm_1\upharpoonright\dom(A_{D,\pm})\big)
\subseteq \ran\big(M_+(z)+M_-(z)\big), \quad z \in\dC_+\cup\dC_-, 
\end{equation}
where $M_\pm$ are (minus) the Dirichlet-to-Neumann maps in \eqref{dnmaps}.

\begin{corollary}\label{utr444}
Let $A_0$ be the orthogonal sum of the Dirichlet operators in \eqref{add}, let $M_\pm$
be the $($minus\,$)$ Dirichlet-to-Neumann maps in \eqref{dnmaps} and let $\gamma$ be the orthogonal
sum of the $\gamma$-fields in \eqref{gammas} $($cf.\ \eqref{gamgam}$)$.
For all $z \in\rho(A)\cap\rho(A_0)$, the resolvent of $A$ is given by
\begin{equation}\label{y65rrf}
(A- z I_{L^2(M)})^{-1}=(A_0- z I_{L^2(M)})^{-1}+\gamma(z)\Theta(z)
\gamma(\ol{z})^*,
\end{equation}
where
\begin{equation}\label{Trafa}
\Theta(z)=\begin{pmatrix} -(M_+(z)+M_-(z))^{-1} &
-(M_+(z)+M_-(z))^{-1} \\ -(M_+(z)+M_-(z))^{-1} &
-(M_+(z)+M_-(z))^{-1} \end{pmatrix}.
\end{equation}
\end{corollary}

In a similar way one obtains a representation for the resolvent of $A$
from Theorem~\ref{resthm2}, where $(A- z I_{L^2(M)})^{-1}$ is compared with the
orthogonal sum of the self-adjoint Dirichlet operator $A_{D,+}$ in $L^2(\Omega_+)$
and the self-adjoint Neumann operator $A_{N,-}$ in $L^2(\Omega_-)$.

\begin{corollary}\label{cor1212}
Let $M_\pm$ be the (minus) Dirichlet-to-Neumann maps in \eqref{dnmaps} and
let $\gamma_\pm$ be the $\gamma$-fields in \eqref{gammas}.
For all $z \in\rho(A)\cap\rho(A_{D,+})\cap\rho(A_{D,-}) \cap\rho(A_{N,-})$, the resolvent of $A$ is given by
\begin{equation}\label{y5rrffa8j}
(A- z I_{L^2(M)})^{-1}=\left(\begin{pmatrix} A_{D,+} & 0 \\ 0 & A_{N,-} \end{pmatrix}
- z I_{L^2(M)}\right)^{-1}
+\widehat\gamma(z)\Sigma(z)\widehat\gamma(\ol{z})^*,
\end{equation}
where
\begin{equation}\label{itrafaf}
\widehat\gamma(z) = \begin{pmatrix} \gamma_+(z) & 0 \\ 0 &
\gamma_-(z)M_-(z)^{-1}\end{pmatrix},  \quad
\Sigma(z) = -\begin{pmatrix} M_+(z) & I_{L^2(\cC)} \\ I_{L^2(\cC)} & -M_-(z)^{-1} \end{pmatrix}^{-1}.
\end{equation}
\end{corollary}

Our next aim is to illustrate the abstract third Green identity from Section~\ref{s4}
in the present context of Schr\"{o}dinger operators on Lipschitz domains on smooth, boundaryless Riemannian manifolds. Since $\dom(A)=H^2(M)$
and the graph norm \eqref{gn} is equivalent to the usual norm on the Sobolev space $H^2(M)$
we have $\sH_2=H^2(M)$ and hence the Gelfand triple in \eqref{gelfand} is of the form
\begin{equation}\label{ju5433}
H^2(M)\hookrightarrow L^2(M)\hookrightarrow H^{-2}(M).
\end{equation}
Using the notation in \eqref{gammagamma} we note that the mappings
\begin{equation}\label{y5rrrafa}
\Upsilon_0=\tau_D^+:H^2(M)\rightarrow L^2(\cC),  \quad
\Upsilon_1=-\tau_N^+:H^2(M)\rightarrow L^2(\cC)
\end{equation}
are bounded, which is clearly in accordance (and also follows from) Lemma~\ref{p3.4}.
Thus, the dual operators
\begin{equation}\label{ut54rafa}
\Upsilon_0^*=(\tau_D^+)^*:L^2(\cC)\rightarrow H^{-2}(M),  \quad
\Upsilon_1^*=(-\tau_N^+)^*:L^2(\cC)\rightarrow H^{-2}(M)
\end{equation}
are bounded. Moreover, under the additional assumption that
\begin{equation}\label{khaba-VVV}
\text{$V \geq 0$ and not identically zero on $M$,}
\end{equation}
it has been proved in \cite[p.~27]{MT00b} that
\begin{equation}\label{khaba-GGG}
-\Delta_g+V:L^2(M)\rightarrow H^{-2}(M) \, \text{ is invertible, with bounded inverse}.
\end{equation}
In such a scenario, if $\cG:H^{-2}(M)\rightarrow L^2(M)$ denotes the inverse of
\eqref{khaba-GGG}, it follows that
\begin{equation}\label{ggg}
\cG(-\Delta_g+V)f=f,\quad f\in\dom(T)=H^{3/2}_\Delta(\Omega_+)\times H^{3/2}_\Delta(\Omega_-).
\end{equation}
Then the (abstract) single and double layer potentials in \eqref{slpot} and \eqref{dlpot} are given by
\begin{equation}\label{slpot2}
\cS:L^2(\cC)\rightarrow L^2(M),\quad\varphi\mapsto\cG\Upsilon_0^*\varphi=\cG(\tau_D^+)^*\varphi,
\end{equation}
and
\begin{equation}\label{dlpot2}
\cD:L^2(\cC)\rightarrow L^2(M),\quad\varphi\mapsto\cG\Upsilon_1^*\varphi=\cG(-\tau_N^+)^*\varphi.
\end{equation}
Moreover, if $E(x,y)$ is the integral kernel of $\cG$, then the action of these abstract single and
double layer potentials on a function $\varphi\in L^2(\cC)$ may be explicitly written as
\begin{equation}\label{slpot2-XXX}
(\cS\varphi)(x)=\int_{\cC}E(x,y)\varphi(y)\,d\sigma_{\!g}(y),\quad x\in M,
\end{equation}
and
\begin{equation}\label{dlpot2-XXX}
(\cD \varphi)(x)=\int_{\cC}\big\langle\mathfrak{n}^{+}(y),{\rm grad}_{g_y}E(x,y)\big\rangle_{T_yM}
\varphi(y)\,d\sigma_{\!g}(y),\quad x\in M,
\end{equation}
which are in agreement with \eqref{slpot2-XXX-YYa}, \eqref{dlpot2-XXX-YYa}.
Furthermore, the abstract jump relations in \eqref{eq:Bracket1}--\eqref{eq:Bracket2} are
\begin{align}
[\Gamma_0f] &= \Gamma_0^+f_+-\Gamma_0^-f_-=\tau_D^+f_+-\tau_D^-f_-,\quad f=(f_+,f_-)^\top\in\dom(T),
\label{eq:Bracket1a}\\
[\Gamma_1f] &= \Gamma_1^+f_++\Gamma_1^-f_-=-\tau_N^+f_+-\tau_N^-f_-,\quad f=(f_+,f_-)^\top\in\dom (T).
\label{eq:Bracket2a}
\end{align}
As a consequence of the above considerations and Theorem~\ref{t3.7} we
obtain the following version of Green's third identity for the Schr\"odinger
operator $-\Delta_g+V$ on $M$.

\begin{theorem}\label{y5art-CCffa}
With the fundamental solution operator $\cG$ in \eqref{ggg}, the layer potentials in
\eqref{slpot2}--\eqref{dlpot2}, and the jump relations in
\eqref{eq:Bracket1a}--\eqref{eq:Bracket2a}, one has
\begin{equation}\label{eq:u2}
f=\cG Tf+\cD[\Gamma_0 f]-\cS[\Gamma_1 f], \quad f \in \dom(T),
\end{equation}
$($i.e., $f=(f_+,f_-)^\top$ with $f_\pm\in H^{3/2}_\Delta(\Omega_\pm)$$)$.
\end{theorem}

In conclusion, we note that boundary triples for elliptic operator in an unbounded external domain
$\Omega_- \subset \bbR^n$, used as an illustration in the introduction,
were studied, for instance, in \cite{BLL13}, \cite{Ma10}. The third Green formula in this situation and its analog in connection with noncompact Riemannian manifolds $M$ requires additional techniques to be discussed elsewhere.

\medskip

\noindent {\bf Acknowledgments.} We are indebted to Mark Ashbaugh, Damir Kinzebulatov, 
and Michael Pang for very helpful discussions.


\end{document}